\documentclass[12pt]{amsart}

\usepackage{amsmath, amsthm, amssymb, amsfonts, enumerate}
\usepackage[colorlinks=true,linkcolor=blue,urlcolor=blue]{hyperref}
\usepackage{dsfont}
\usepackage{color}
\usepackage{geometry}
\usepackage{todonotes}

\usepackage{graphicx}
\usepackage{caption}

\usepackage{float}

\usepackage{mathtools}
\usepackage{soul}

\mathtoolsset{showonlyrefs}

\def \cA{\mathcal{A}}
\def \cC{\mathcal{C}}
\def \cD{\mathcal{D}}

\def \cH{\mathcal{H}}
\def \cF{\mathcal{F}}
\def \cG{\mathcal{G}}

\def \cJ{\mathcal{J}}
\def \cL{\mathcal{L}}

\def \cO{\mathcal{O}}

\def \P{\mathsf P}
\def \Q{\mathsf Q}
\def \M{\mathsf M}

\def \E{\mathsf E}

\def \N{\mathbb{N}}
\def \R{\mathbb{R}}

\def \ud{\mathrm{d}}
\def \e{\mathrm{e}}

\newcommand{\eps}{\varepsilon}

\newtheorem{theorem}{Theorem}[section]
\newtheorem{lemma}[theorem]{Lemma}
\newtheorem{corollary}[theorem]{Corollary}
\newtheorem{proposition}[theorem]{Proposition}
\newtheorem{example}[theorem]{Example}
\newtheorem{definition}[theorem]{Definition}
\newtheorem{remark}[theorem]{Remark}
\newtheorem{assumption}[theorem]{Assumption}

\title[Stochastic control with state constraint]{A class of stochastic control problems\\ with state constraints}
\author[De Angelis]{Tiziano De Angelis}
\author[Ekstr\"om]{Erik Ekstr\"om}
\address{T.\ De Angelis: School of Management and Economics, Dept.\ ESOMAS, University of Torino, C.so Unione Sovetica 218bis, 10134, Torino, Italy; Collegio Carlo Alberto, Piazza Arbarello 8, 10122, Torino, Italy}
\email{tiziano.deangelis@unito.it}
\address{E.\ Ekstr\"om: Department of Mathematics, Uppsala University, Box 256, 75105 Uppsala, Sweden}
\email{ekstrom@math.uu.se}
\date{\today}

\numberwithin{equation}{section}

\begin{document}

\begin{abstract}
We obtain a probabilistic solution to linear-quadratic optimal control problems with state constraints. Given a closed set $\cD\subseteq [0,T]\times\R^d$, a diffusion $X$ in $\R^d$ must be linearly controlled in order to keep the time-space process $(t,X_t)$ inside the set $\cC\coloneqq([0,T]\times\R^d)\setminus\cD$, while at the same time minimising an expected cost that depends on the state $(t,X_t)$ and is quadratic in the speed of the control exerted. 
We find a probabilistic representation for the value function and an optimal control under a set of mild sufficient conditions concerning the coefficients of the underlying dynamics and the regularity of the set $\cD$. The optimally controlled dynamics is in strong form, in the sense that it is adapted to the filtration generated by the driving Brownian motion. Fully explicit formulae are presented in some relevant examples.   
\end{abstract}

\maketitle

\section{Introduction}
Problems of optimal navigation of a vehicle are prominent in the engineering literature and they can be cast as stochastic control problems with constrained state-dynamics (see, e.g., Shah et al.\ \cite{STP}, where a linear-quadratic structure is considered). The constraint imposes that the controlled process {\em must not} visit certain regions of the state-space, because that would correspond to colliding with other physical objects. Motivated by this type of questions we find a {\em probabilistic solution} to linear-quadratic stochastic control problems with state-constraints. Given a closed set $\cD\subseteq [0,T]\times\R^d$, a diffusion $X$ in $\R^d$ must be linearly controlled in order to keep the time-space process $(t,X_t)$ inside the set $\cC\coloneqq([0,T]\times\R^d)\setminus\cD$, while at the same time minimising an expected cost that depends on the state $(t,X_t)$ and is quadratic in the speed of the control exerted.

The problem we study has 
many connections to other questions in stochastic control and stochastic analysis. For example, we find links to the theory of logarithmic transformations and risk-sensitive optimisation (see Fleming and Soner \cite[Ch.\ VI]{FS} and our Remark \ref{rem:risksensitive}), exit problems in the small-noise limit (see, e.g., Fleming \cite{F}), viability problems (see, e.g., Bardi and Jensen \cite{BJ}, Buckdahn et al. \cite{BFQ, BPQR, BQRR}), Doob's $h$-transform and potential theory
(see, e.g., Rogers and Williams \cite{RW1,RW2} and our Remark \ref{rem:doob}), and certain stochastic target problems (see, e.g., Ankirchner et al.\ \cite{AK,AJK}, Bank and Voss \cite{BV}, Dolinsky et al.\ \cite{DGGG}, Horst and Xia \cite{HX}). 

In viability problems the aim is to determine conditions under which it is possible to keep a suitably controlled dynamics inside a given domain. The question can be addressed by setting up a stochastic control problem whose cost function depends on the distance of the controlled dynamics from the boundary of the domain. Papers on this topic mentioned above use methods based on viscosity solutions of suitable Hamilton-Jacobi-Bellman equations. 

Stochastic target problems concern controlling a state dynamics in order to minimise (or maximise) a cost function (or reward) under the constraint that the terminal value of the dynamics lies in a certain set. The methods involved in the solution of such problems rely mostly on Backward Stochastic Differential Equations and, in the linear-quadratic case, on Riccati equations. Fleming and Sheu \cite{FSh} formulated linear-quadratic control problems with state constraint at the terminal time (i.e., in our notation $\cD=\{T\}\times\Sigma$ for $\Sigma\subset\R^d$) and then focussed their attention on the particular case of so-called {\em bridges} (i.e., $\Sigma=\{x_0\}$). In a broader sense, Dai Pra's work \cite{DP} can also be interpreted as a linear-quadratic stochastic target problem: the optimiser minimises a quadratic cost of exerting control with the constraint that the distribution of the final state of the dynamics is fixed a priori. This was motivated by connections with reciprocal diffusions (Bernstein \cite{B}) and the so-called {\em Schr\"odinger bridge}. 

The closest works to ours that we could find in the literature are a paper by Day \cite{D} and a paper by Fuhrman \cite{F03}. In \cite{D}, the author considers an infinite-time horizon, linear-quadratic, stochastic control problem in which the state dynamics is constrained to evolve in a given domain $\Delta\subset\R^d$. Day's work was motivated by earlier results by Fleming \cite{F} on exit probabilities and large deviation results (in the small-noise limit). 
We share with the approach by Day \cite{D} the use of a logarithmic transformation already introduced by Fleming \cite{F} and widely illustrated in \cite[Ch.\ VI]{FS} (the same transformation is also used in \cite{DP}, combined with probabilistic methods).
Differently from our probabilistic approach, Day uses PDE methods for second order elliptic equations to solve his problem. He imposes that the domain $\Delta$ be bounded with $C^2$ boundary and the diffusion coefficient of the controlled dynamics be uniformly elliptic. Both drift and diffusion coefficients are assumed Lipschitz continuous. 

In our work we consider a finite-time horizon, leading to associated parabolic equations. We also substantially relax the requirement on the smoothness of the boundary of the domain $\cC$ by using the notion from potential theory of ``regularity in the sense of diffusions'' \cite[Ch.\ I.11]{BG}. It is worth noticing (see Remark \ref{rem:suff}) that such regularity can be obtained in many examples without smoothness of the boundary. Also (global) Lipschitz continuity of the coefficients and uniform ellipticity are not necessary, as long as certain boundary value problems associated with the infinitesimal generator of the uncontrolled dynamics admit a classical solution on smooth bounded (parabolic) domains; see our Assumption \ref{ass:L} and the subsequent discussion for details. 

Furhman \cite{F03} instead takes a BSDE approach to solve linear-quadratic problems when the underlying dynamics is infinite-dimensional. Also in this case, Fleming's logarithmic transformation is a key tool in the analysis. The assumptions on the payoff functions are general enough to allow for a formulation with state constraints which formally includes our finite-dimensional problems as a special case (see in particular \cite[Sec.\ 6]{F03}). However, due to the generality of the setting, an optimally controlled (constrained) dynamics in \cite{F03} can only be constructed in weak form with a limiting argument, whereas we construct it in strong form. Differently from our paper, the analysis of the regularity of the value function and solvability of the associated Hamilton-Jacobi-Bellman (HJB) equation are not addressed in \cite{F03}, where the HJB and the form of the optimal control are only outlined in the final paragraph of \cite[Sec.\ 6]{F03}. 

In our setting, using ideas from diffusion theory and It\^o's calculus we derive a closed-form expression for the value function $v:[0,T]\times\R^d\to [0,\infty]$ of the constrained control problem (notice that $v=\infty$ in $\cD$). We also show that $v$ is continuously differentiable in time and twice continuously differentiable in space in the set $\cC$. That allows us to prove that $v$ solves a suitable HJB equation in $\cC$ with singular boundary conditions. Moreover, thanks to the regularity of $v$ we construct an optimal Markovian control, {\em in closed form}, that minimises the expected cost and satisfies the constraint on the controlled dynamics. Our optimally controlled dynamics is expressed in strong formulation, in the sense that it is adapted to the Brownian filtration. This fact is not trivial, because the optimal control does not enjoy linear growth properties and it actually blows up at the boundary of the set $\cC$. The closed-form expression for $v$ is of a purely probabilistic nature. 
We show that $v=-2\ln u$, where $u:[0,T]\times\R^d\to [0,1]$ is the expectation of an exponential payoff depending on an unconstrained process $Z$ killed upon entering the set $\cD$. This representation can be used to obtain explicit formulae for $v$ and for the optimal control, when the density of the killed process is known explicitly. Moreover, it can be used for numerical simulation with simple Monte Carlo methods in all other situations. Finally, we recover the classical results from risk-sensitive literature (e.g., \cite{BD}) when we set $\cD=\varnothing$ (i.e., the dynamics is unconstrained; see Remark \ref{rem:risksensitive}), and the classical Doob's $h$-transform when the cost function does not depend explicitly on the state $X$ (see Remark \ref{rem:doob}). 

The paper is organised as follows. In Section \ref{sec:setup} we set up the problem, state the main result (Theorem \ref{thm:main}) and provide explicitly solvable examples as an illustration. We also discuss in detail the connections to risk-sensitive functionals and Doob's $h$-transform. In Section \ref{sec:proof} we prove our main result. In Section \ref{sec:suffcond} we provide simple sufficient conditions that imply the validity of a general assumption made in Theorem \ref{thm:main}. 

\section{Setting and main results}\label{sec:setup}
Let $d,d'\ge 1$, and denote $\cO\coloneqq[0,T]\times\R^d$. Given a subset $A\subseteq \cO$ we denote $\overline A$ the closure of $A$ in $\cO$. The class $C(A)$ denotes continuous functions $A\to\R$. For $j,k\in\N\cup\{0\}$ the class
$C^{j,k}(A)$ denotes the subset of $C(A)$ consisting of functions that are $j$-times continuously differentiable in $t$ and $k$-times in $x$. 
We denote 
$A_{[s,t)}=A\cap([s,t)\times\R^d)$ and $A_{t}=A\cap(\{t\}\times\R^d)$ for $0\le s<t\le T$.
In particular, $\cO_{[0,T)}\coloneqq[0,T)\times\R^d$. Finally, to measure the distance from a point $(t,x)\in\cO$ to a set $A\subset\cO$ we use 
\[
\mathrm{dist}((t,x),A)\coloneqq\inf\{|(t,x)-(s,y)|_{d+1},\ (s,y)\in A\},
\]
where $|\cdot|_{d+1}$ is the Euclidean norm in $\R^{d+1}$. 

Let $\cD$ be a closed subset of $\cO$ and let $\cC=\cO\setminus \cD$ be its complement. 
The set $\cD$ represents the ``forbidden'' region for the controlled dynamics. The case of interest for us is when $\cD$, $\cC$ and $\cC_T$ are all 
non-empty; cf. Assumption \ref{ass:ZD}. Remark \ref{rem:risksensitive} below illustrates the case $\cD=\varnothing$.
Simple examples of sets $\cD$ are $\{T\}\times (\cup_{i=1}^d \{x_i\le 0\} )$ and $\cD=[0,T]\times
(\cup_{i=1}^d \{x_i\le 0\} )$ (cf. Examples~\ref{ex:1} and \ref{ex:2} below, respectively). 
 
In order to introduce the controlled dynamics we take measurable functions  
$\mu:\cO\to\R^d$ and $\sigma:\cO\to \R^{d\times d'}$,
which are locally Lipschitz in $x$ and of linear growth, uniformly for $t\in [0,T]$.
We are going to formulate our state-constrained stochastic control problem in a {\em strong} form on a given filtered probability space $(\Omega,\cF,(\cF_t)_{t\in[0,T]},\P)$ equipped with a $d'$-dimensional Brownian motion $(W_t)_{t\in[0,T]}$. With no loss of generality we will assume that $(\cF_t)_{t\in[0,T]}$ is the Brownian filtration augmented with $\P$-null sets.
\begin{definition}[{{\bf Admissible controls}}]\label{def:admissible}
Fix $(t,x)\in\cO$. 
A {\em control} is a $\R^{d'}$-valued process $(a_s)_{s\in[t,T]}$ which is progressively measurable 
and such that the SDE 
\begin{equation}\label{eq:Xa}
X_s=x+\int_t^s \big[\mu(u,X_u)+\sigma(u,X_u)a_u\big]\ud u+\int_t^s \sigma(u,X_u)\ud W_u,\quad s\in[t,T],
\end{equation}
admits a unique {\em strong solution} $(X^{t,x}_s)_{s\in[t,T]}$. 
Moreover, for $(t,x)\in\cC$ the control $(a_s)_{s\in[t,T]}$ is {\em admissible} if 
\[
\P\big((s,X^{t,x}_s)\in \cC,\,\forall s\in[t,T]\big)=1.
\]
The class of admissible controls is denoted by $\cA^{\cD}_{t,x}$.
Finally, an admissible control is {\em Markovian} if $a_u= \alpha(u,X_u)$, $u\in[t,T]$, where $\alpha:[0,T]\times\R^d\to \R^{d'}$ is a measurable function. 
\end{definition}

\begin{remark}\label{rem:yamada}
A technical hurdle with admissible controls in strong form arises because such controls may be unbounded and exhibit singularities of the trajectories near the boundary of the set $\cC$. For the same reasons even a weak formulation of the stochastic control problem via change of measure and Girsanov theorem (see, e.g., \cite[Sec.\ 2]{KZ}) would lead to difficulties in the definition of the exponential martingale, which is needed as Radon-Nikodym derivative for the measure change.
\end{remark}

\begin{remark}
Sometimes we use the notation $X^{t,x;a}$ to emphasise the dependence of $X$ on the control process $(a_s)_{s\in[t,T]}$ and on the initial condition $(t,x)$. 
A priori we do not assume that $X^{t,x;a}$ be Markovian but it will turn out that the optimally controlled dynamics is such.
\end{remark}

In order to set up our problem we introduce measurable functions 
$f:\cO\to [0,\infty)$ and $g:\R^d\to[0,\infty)$.
Given $(t,x)\in\cC$ and an admissible control $a\in\cA^{\cD}_{t,x}$, the associated cost function is denoted 
\begin{equation}\label{eq:J}
\cJ_{t,x}(a)\coloneqq\E\left[\int_t^{T}\!\!\big(f(s,X^{t,x;a}_s)+|a_s|^2_{d'}\big)\ud s+g(X^{t,x;a}_T)\right],
\end{equation}
where $|\,\cdot\,|_{d'}$ stands for the Euclidean norm in $\R^{d'}$ and $\E$ is the expectation under the measure $\P$. Later on we will also use $|\cdot|_{d\times d'}$ for the Euclidean norm on the space of $d\times d'$-matrices, i.e., $|A|_{d\times d'}=(\sum_{i=1}^d\sum_{j=1}^{d'}|A_{ij}|^2)^\frac12$. We will omit the subscript and only use $|\cdot|$ to indicate the Euclidean norm in either $\R^d$, $\R^{d'}$ or $\R^{d\times d'}$, when no confusion shall arise. The value function of the stochastic control problem with state constraint reads
\begin{equation}\label{eq:P1}
v(t,x)=\inf_{a\in\cA^{\cD}_{t,x}}\cJ_{t,x}(a),
\end{equation}
for $(t,x)\in\cC$, and we set $v\equiv +\infty$ on $\cD$.

\begin{remark}
The question as to whether the set $\cA^\cD_{t,x}$ is empty or not is a so-called {\em viability} problem for the set $\cC$. Such problems have been studied extensively in the literature (see, e.g., \cite{BJ,BFQ,BPQR,BQRR}).
Conditions are provided in Theorem~\ref{thm:main} under which $\cA^\cD_{t,x}\neq\varnothing$ and an optimal control for our stochastic control problem is obtained.
\end{remark}

Our main result is stated in terms of an auxiliary uncontrolled dynamics. We define it on another probability space, with no loss of generality, in order to emphasise that the link between the problem in \eqref{eq:P1} and the auxiliary dynamics is purely via the laws of the processes. We consider a filtered probability space $(\Omega',\cG,(\cG_s)_{s\in[0,T]},\Q)$ equipped with a $d'$-dimensional Brownian motion $(\bar W_t)_{t\in[0,T]}$. Thanks to the assumed regularity of $\mu$ and $\sigma$, for any $(t,z)\in\cO$ the SDE
\begin{equation}\label{eq:Z}
Z_s=z+\int_t^s \mu(u,Z_u)\ud u+\int_t^s \sigma(u,Z_u)\ud \bar W_u,\quad s\in[t,T],
\end{equation}
admits a unique strong solution.  
Set 
\begin{equation}\label{eq:tauD}
\tau_\cD=\inf\{s\in [t,T] : (s,Z_s)\in \cD\}=\inf\{s\in[t,T]:(s,Z_s)\notin\cC\},
\end{equation}
with the convention that $\inf\varnothing=\infty$. 
Sometimes we use the notation $Z^{t,z}$ and $\tau^{t,z}_\cD$ to keep track of the initial conditions. The process $Z$ is strong Markov with respect to its own filtration and we use the notation $\Q_{t,z}(\,\cdot\,)=\Q(\,\cdot\,|Z_t=z)$. Expectation under the measure $\Q$ will be denoted $\E^\Q$. We also assume that 
$\Q_{t,z}(\tau_{\cD}=\infty)>0$ for any starting point $(t,z)\in\cC$ so that
the process $(s,Z_s)$ remains in $\cC$ during $[t,T]$ with positive probability.

\begin{assumption}\label{ass:ZD}
For any $(t,z)\in\cC$ it holds $\Q_{t,z}(\tau_\cD = \infty)>0$.
\end{assumption}
 
Assumption \ref{ass:ZD} holds in $d=1$ if, for example, the two conditions below are satisfied
(the extension to $d>1$ is straightfoward for diffusion processes whose law has full support): 
\begin{itemize}
\item[(i)] The functions $\sigma$, $\sigma^{-1}$ and $\mu$ are continuous on $\overline\cC$;
\item[(ii)] For any starting point
$(t,z)\in\cC$ there is a ``tunnel''
\[
\mathcal E_{t,z}\coloneqq\{(s,y)\in[t,T]\times\R :b(s)<y<c(s)\}\subset \cC,
\]
where $b$ and $c$ are continuous functions with $b(t)<z<c(t)$ and 
\[
\inf_{s\in[t,T]}\big( c(s)-b(s)\big)>0.
\]
\end{itemize}

The infinitesimal generator of the process $Z$ is given by
\begin{equation}
(\cL \psi)(t,z)\coloneqq\tfrac{1}{2}\mathrm{tr}(\sigma\sigma^\top D^2 \psi)(t,z)+\langle \mu,\nabla \psi \rangle(t,z),\quad \psi\in C^{0,2}(\cO),
\end{equation}
where $\mathrm{tr}(\,\cdot\,)$ is the trace, $\sigma^\top$ is the transpose of $\sigma$, $\nabla \psi$ and $D^2 \psi$ are the spatial gradient and the spatial Hessian matrix of $\psi$, respectively, and $\langle\,\cdot\,,\,\cdot\,\rangle$ is the scalar product in $\R^{d}$.
A technical assumption concerning the generator $\cL$ is stated below.
We use the notations $B\subset\R^d$ for a generic open ball and $U_{t_1,t_2}^B\coloneqq [t_1,t_2)\times B$ for any given $0\le t_1<t_2\le T$. The parabolic boundary of $U\coloneqq U^B_{t_1,t_2}$ is denoted $\partial_P U\coloneqq\big([t_1,t_2]\times\partial B\big)\cup \big(\{t_2\}\times B\big)$, where $\partial B= \overline B\setminus B$.

In what follows, with a small abuse of notation we consider $g$ as a function defined on $\cO$ and, in particular, we consider its restriction to $\cC_T$.

\begin{assumption}\label{ass:L}
The function $g$ is continuous on $\cC_T$. The functions $\mu$, $\sigma$ and $f$ are such that, given any set $U=U_{t_1,t_2}^B$ with closure contained in $\cC$, and any continuous function $\varphi:\partial_P U\to \R$, there is a unique classical solution $w\in C(\overline U)\cap C^{1,2}(U)$ of the boundary value problem
\begin{equation}
\left\{\begin{array}{ll}
\partial_t w(t,z)+(\cL w)(t,z)-\frac{1}{2}f(t,z)w(t,z)=0,& (t,z)\in U,\\[+5pt]
w(t,z)=\varphi(t,z),& (t,z)\in \partial_P U.
\end{array}\right.
\end{equation}
\end{assumption}

Sufficient conditions that imply the validity of the assumption above are for example, uniform non-degeneracy of the diffusion coefficient on any compact and local H\"older-continuity of functions $\mu$, $\sigma$ and $f$ (see, e.g., \cite[Ch.\ 3, Sec.\ 4, Thm.\ 9]{Fri}). 
However, weaker conditions are also known in some cases of degenerate diffusion coefficient, leading to $\cL$ being a so-called hypoelliptic operator (see, e.g., \cite{Hor}, and \cite{LP94} for evolution equations).  

To state the main result of the paper, we define a function $u:\cO\to[0,1]$ as
\begin{equation}\label{eq:u}
u(t,z)\coloneqq\E_{t,z}^\Q\Big[\exp\Big(-\frac{1}{2}\int_t^T f(s,Z_s)\ud s-\frac{1}{2}g(Z_T)\Big)1_{\{T<\tau_\cD\}}\Big].
\end{equation}
In order to state our final assumption we introduce sets 
\begin{equation}\label{eq:setnot}
\begin{aligned}
\cD^\circ&\coloneqq \mathrm{int}(\cD),\quad \cD^\circ_{[0,T)}\coloneqq \cD^\circ\cap([0,T)\times\R^d),\quad \overline \cC_{[0,T)}\coloneqq\overline\cC\cap([0,T)\times\R^d),\\ 
\overline{\cC_T}&\coloneqq \overline{\cC\cap(\{T\}\times\R^d)},\quad \cD^\circ_T\coloneqq\cD_T\setminus\overline{\cC_T}\quad\text{and}\quad\partial\cC_T\coloneqq \cD_T\cap\overline{\cC_T}.
\end{aligned}
\end{equation}
Notice that $\cD^\circ_T\neq (\cD^\circ)_T$ because $(\cD^\circ)_T=\cD^\circ\cap(\{T\}\times\R^d)=\varnothing$ whereas $\cD^\circ_T$ may be nonempty.
\begin{assumption}\label{ass:u} 
The following holds:
\begin{itemize}
\item[(i)] The function $u$ is continuous\footnote{Because $u=0$ on $\cD$ it is equivalent to require continuity of $u$ on $\overline\cC_{[0,T)}\cup\cC_T\cup\cD^\circ_T$} on $\cO\setminus\partial\cC_T$. 
\item[(ii)] The process $Z$ is such that 
\[
\Q_{t,z}\big((T,Z_T)\in\partial\cC_T\big)=0,\quad \text{for $(t,z)\in\cO_{[0,T)}$}.
\]
\item[(iii)] On the event $\{T<\tau_\cD\}$ we have $\int_t^T f(s,Z_s)\ud s<\infty$, $\Q_{t,z}$-a.s.
\end{itemize}
\end{assumption}
It may be worth noticing that the above assumption allows for discontinuities of $u$ at corner points of $\cC$, i.e., $u$ may be discontinuous at points in $\partial\cC_T$. However, the process $Z$ does not visit those points with probability one. That is for a good reason, as illustrated in Examples \ref{ex:1}--\ref{ex:3} and in Section \ref{sec:suffcond}.
The condition $\Q_{t,z}((T,Z_T)\in\partial\cC_T)=0$ holds if for example $Z_T$ has a density with respect to the Lebesgue measure and $\partial\cC_T$ is a null set for such measure.

\begin{theorem}\label{thm:main}
Let Assumptions \ref{ass:ZD}, \ref{ass:L} and \ref{ass:u} hold. 
Then $u\in C^{1,2}(\cC_{[0,T)})$ with $u>0$ in $\cC$, $u=0$ in $\cD$ and the value function of problem \eqref{eq:P1} reads
\[
v(t,x)=-2\ln u(t,x),\quad\text{for $(t,x)\in\cO$}.
\] 
Setting 
\begin{equation}\label{eq:a*}
\alpha^*(t,x)\coloneqq \left\{
\begin{array}{cl}
-\frac{1}{2}\sigma^\top(t,x)\nabla u(t,x)/u(t,x),& (t,x)\in \cC_{[0,T)},\\
0, & (t,x)\notin \cC_{[0,T)},
\end{array}
\right.
\end{equation}
the SDE
\begin{equation}\label{eq:X*}
X^*_s=x+\int_t^s [\mu(u,X^*_u)+\sigma(u,X^*_u)\alpha^*(u,X^*_u)]\ud u+\int_t^s \sigma(u,X^*_u)\ud W_u, \quad s\in[t,T],
\end{equation}
admits a unique strong solution $X^*$ for any $(t,x)\in \cC$.
Moreover, $X^*$ is an optimally controlled dynamics in $\cC$ $($i.e, $(s,X^*_s)\in\cC$ for all $s\in[t,T]$, $\P$-a.s.\ and $a^*_s\coloneqq\alpha^*(s,X^*_s)$ is a minimiser in \eqref{eq:P1}).
\end{theorem}

The proof of the theorem will be distilled in a few technical steps in Section \ref{sec:proof}. Sufficient conditions for the continuity of the function $u$ required in Assumption \ref{ass:u} are provided in Proposition \ref{prop:suffcond} in Section \ref{sec:suffcond}. By definition, $u=0$ on $\cD$ and therefore $v=+\infty$ on $\cD$. In the process of proving the theorem we also show that $v$ is the classical solution of a suitable HJB equation (cf.\ Corollary \ref{cor:V}).

Now we illustrate some explicit solutions. We emphasise that the value functions we obtain do not have a separable structure and it would therefore be difficult to solve directly the Hamilton-Jacobi-Bellman equation via an educated guess. In the examples we denote 
\[
\varphi(z)=(2\pi)^{-1/2}\exp(-\tfrac{1}{2}z^2)\quad\text{and}\quad \Phi(x)=\int_{-\infty}^x\varphi(z)\ud z.
\]
Throughout the next three examples we assume $\cO=[0,T]\times\R$, $\sigma\equiv 1$ and  $\mu\equiv f\equiv g\equiv 0$.

\begin{example}\label{ex:1}
For $\cD\coloneqq\{T\}\times(-\infty,0]$ the function $u$ from \eqref{eq:u} takes the simple form 
\[
u(t,x)=\Q(x+\bar W_{T-t} > 0)=\Phi\Big(\frac{x}{\sqrt{T-t}}\Big),\quad (t,x)\in\cO_{[0,T)},
\] 
and $u(T,x)=1_{(0,\infty)}(x)$.
Assumptions in Theorem \ref{thm:main} hold (notice the discontinuity of $u$ at $(t,x)=(T,0)$) and the value function in \eqref{eq:P1} reads
\[
v(t,x)=-2\ln \Phi\Big(\frac{x}{\sqrt{T-t}}\Big),\quad (t,x)\in\cO_{[0,T)},
\] 
with $v(T,x)=+\infty 1_{(-\infty,0]}(x)$, $x\in\R$.
Moreover, an optimal Markovian control is given by
\[
\alpha^*(t,x)=\frac{1}{\sqrt{T-t}}\frac{\varphi\Big(\frac{x}{\sqrt{T-t}}\Big)}{\Phi\Big(\frac{x}{\sqrt{T-t}}\Big)},\quad (t,x)\in\cO_{[0,T)}.
\]
\end{example}

\begin{figure}
\centering
\includegraphics[scale=.4]{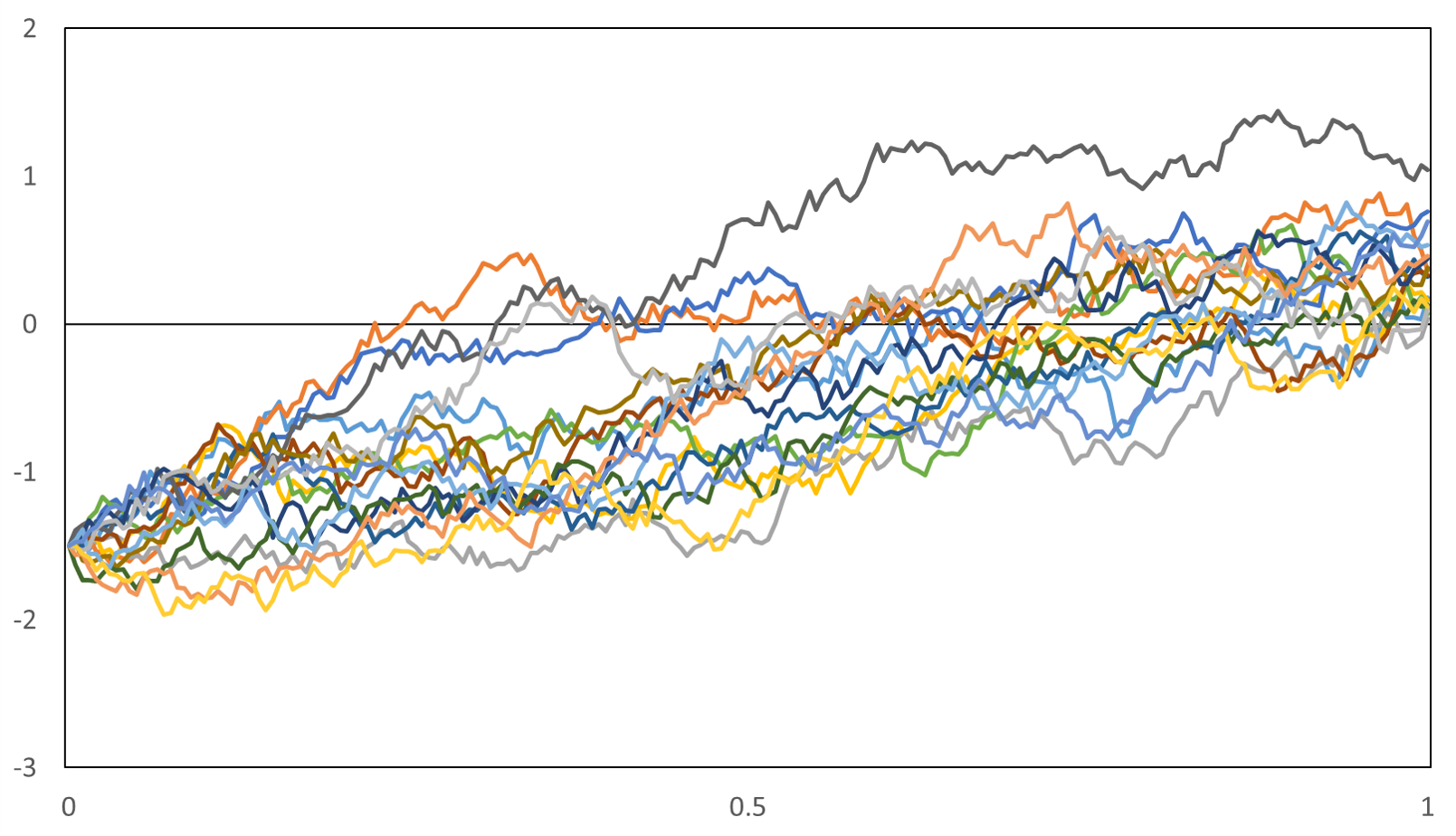}
\caption{Optimal trajectories for Example \ref{ex:1} with $T=1$, obtained by standard Euler-Maruyama method with $X_0^*=-1.5$ and time-step $\Delta t=0.005$.}
\end{figure}

\begin{example}\label{ex:2}
Take $\cD\coloneqq[0,T]\times(-\infty,0]$ and set
$\tau^x_0\coloneqq\inf\{s\ge 0 : x+\bar W_s\le 0\}$.
The function $u$ from \eqref{eq:u} reads $u(T,x)=1_{(0,\infty)}(x)$ for $x\in\R$, $u(t,x)=0$ for $(t,x)\in\cD_{[0,T)}$ and
\begin{eqnarray*}
u(t,x)&=&\Q(\tau^x_0>T-t) \\
&=& \Q\Big(\sup_{0\le s\le T-t}(-\bar W_s)<x\Big)=\Q(|\bar W_{T-t}|<x)\\
&=& 2\int_0^{x/\sqrt{T-t}}\varphi(z)\ud z=2\Phi\Big(\frac{x}{\sqrt{T-t}}\Big)-1,\quad (t,x)\in\cC_{[0,T)},
\end{eqnarray*}
where we used the well-known equality in law $|\bar W_{T-t}|\,{\buildrel d \over =}\,\sup_{0\le s\le T-t}(-\bar W_s)$. Also in this case the assumptions of Theorem~\ref{thm:main} hold (notice the discontinuity of $u$ at $(t,x)=(T,0)$) so the value of problem \eqref{eq:P1} reads
\[
v(t,x)=-2\ln\Big(2\Phi\Big(\frac{x}{\sqrt{T-t}}\Big)-1\Big),\quad (t,x)\in\cC_{[0,T)},
\] 
with $v(T,x)=+\infty 1_{(-\infty,0]}(x)$ for $x\in\R$ and $v(t,x)=+\infty$ for $(t,x)\in\cD_{[0,T)}$.
An optimal Markovian control reads 
\[
\alpha^*(t,x)=\frac{2}{\sqrt{T-t}}\frac{\varphi\Big(\frac{x}{\sqrt{T-t}}\Big)}{2\Phi\Big(\frac{x}{\sqrt{T-t}}\Big)-1},\quad (t,x)\in\cC_{[0,T)}.
\]
\end{example}

\begin{figure}
\centering
\includegraphics[scale=.4]{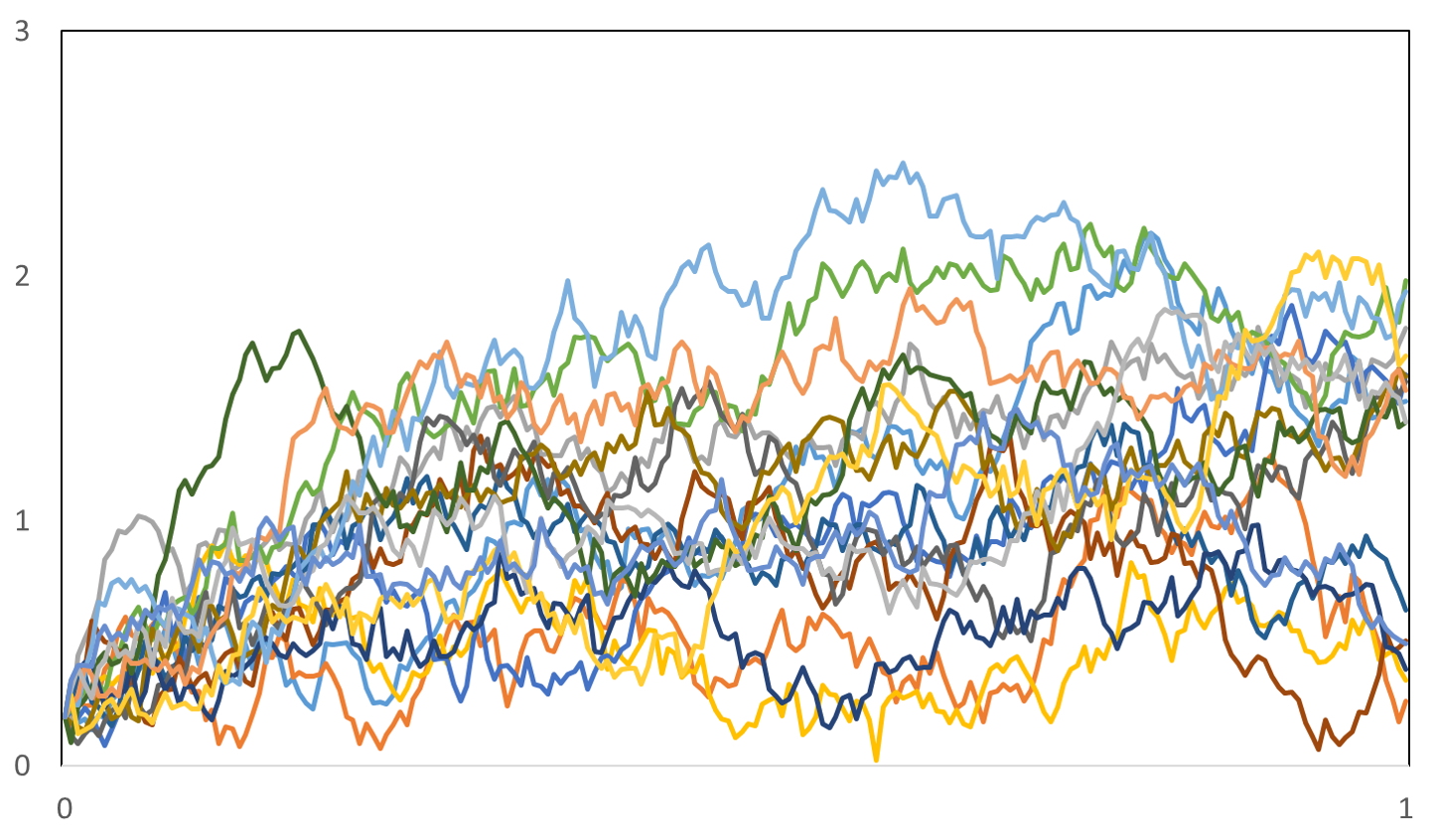}
\caption{Optimal trajectories for Example \ref{ex:2} with $T=1$, obtained by standard Euler-Maruyama method with $X_0^*=0.2$ and time-step $\Delta t=0.005$.}
\end{figure}

\begin{example}\label{ex:3}
In this final example we wish to illustrate a situation in which Assumption \ref{ass:u} is not immediately met but there is nevertheless a reduction of the problem to one that fits our framework. Let $\cD\coloneqq\{t_0\}\times[x_0,x_1]$ for some $t_0\in(0,T)$ and $x_0<x_1$. The function $u$ from \eqref{eq:u} can be computed as
\[
u(t,x)=
\left\{\begin{array}{cl}
1 & t>t_0,\ x\in\R,\\
1-1_{[x_0,x_1]}(x) & t=t_0,\ x\in\R, \\
\Phi\Big(\frac{x-x_1}{\sqrt{t_0-t}}\Big) + \Phi\Big(\frac{x_0-x}{\sqrt{t_0-t}}\Big) & t< t_0,\ x\in\R.
\end{array}
\right.
\]
Clearly $u\in C(\cC)$ but it fails to be continuous on $\overline\cC_{[0,T)}$. It is however important to notice that in the original control problem there is no need to exert control after time $t_0$. Then, $v(t,x)=0$ for $(t,x)\in(t_0,T]\times\R$. This indicates that we can restrict the problem to $[0,t_0]\times\R$ by simply replacing $T$ by $t_0$ in \eqref{eq:J}--\eqref{eq:P1}. By the same reasoning $\{T<\tau_\cD\}=\{t_0<\tau_\cD\}$ and we can replace the indicator function in the definition of $u$ in \eqref{eq:u}, restricting our analysis to $[0,t_0]\times\R$. This reduction puts us formally back into a framework that fits Assumption \ref{ass:u} because $u\in C([0,t_0]\times(\R\setminus\{x_0\}\cup\{x_1\}))$.
Theorem~\ref{thm:main} applies and yields $v=-2\ln u$. An optimal Markovian control reads as 
\[
\alpha^*(t,x)=\frac{1}{\sqrt{t_0-t}}\frac{\varphi\Big(\frac{x-x_1}{\sqrt{t_0-t}}\Big) - \varphi\Big(\frac{x_0-x}{\sqrt{t_0-t}}\Big)}{\Phi\Big(\frac{x-x_1}{\sqrt{t_0-t}}\Big) + \Phi\Big(\frac{x_0-x}{\sqrt{t_0-t}}\Big)}, \quad (t,x)\in[0,t_0)\times\R.
\]
\end{example}

The last example illustrates a methodology that could be extended to more general situations by the simple heuristic observation that, for any $t_0\in[t,T)$, the Dynamic Programming Principle yields
\[
v(t,x)=\inf_{a\in\cA^{\cD}_{t,x}}\E\Big[\int_t^{t_0}\big(f(s,X^{t,x;a}_s)+\big|a_s\big|^2_{d'}\big)\ud s+v(t_0,X^{t,x;a}_{t_0})\Big].
\]
The associated function $u$ reads
\[
u(t,z)=\E_{t,z}\Big[\exp\Big(-\frac12\int_t^{t_0}f(s,Z_s)\ud s-\frac12 v(t_0,Z_{t_0})\Big)1_{\{t_0<\tau_{\cD_0}\}}\Big],
\]
where $\tau_{\cD_0}=\inf\{s\in[t,t_0]:(s,Z_s)\in\cD\}$, with $\inf\varnothing =+\infty$. If for some reason the function $v(t_0,x)$ is known (e.g., it is obtained by solving the problem on $(t_0,T]\times\R^d$) and the function $u$ defined above satisfies Assumption \ref{ass:u} on the restricted domain $[0,t_0]\times\R^d$, then we can solve the problem on that domain using Theorem \ref{thm:main}.

\begin{figure}
\centering
\includegraphics[scale=.4]{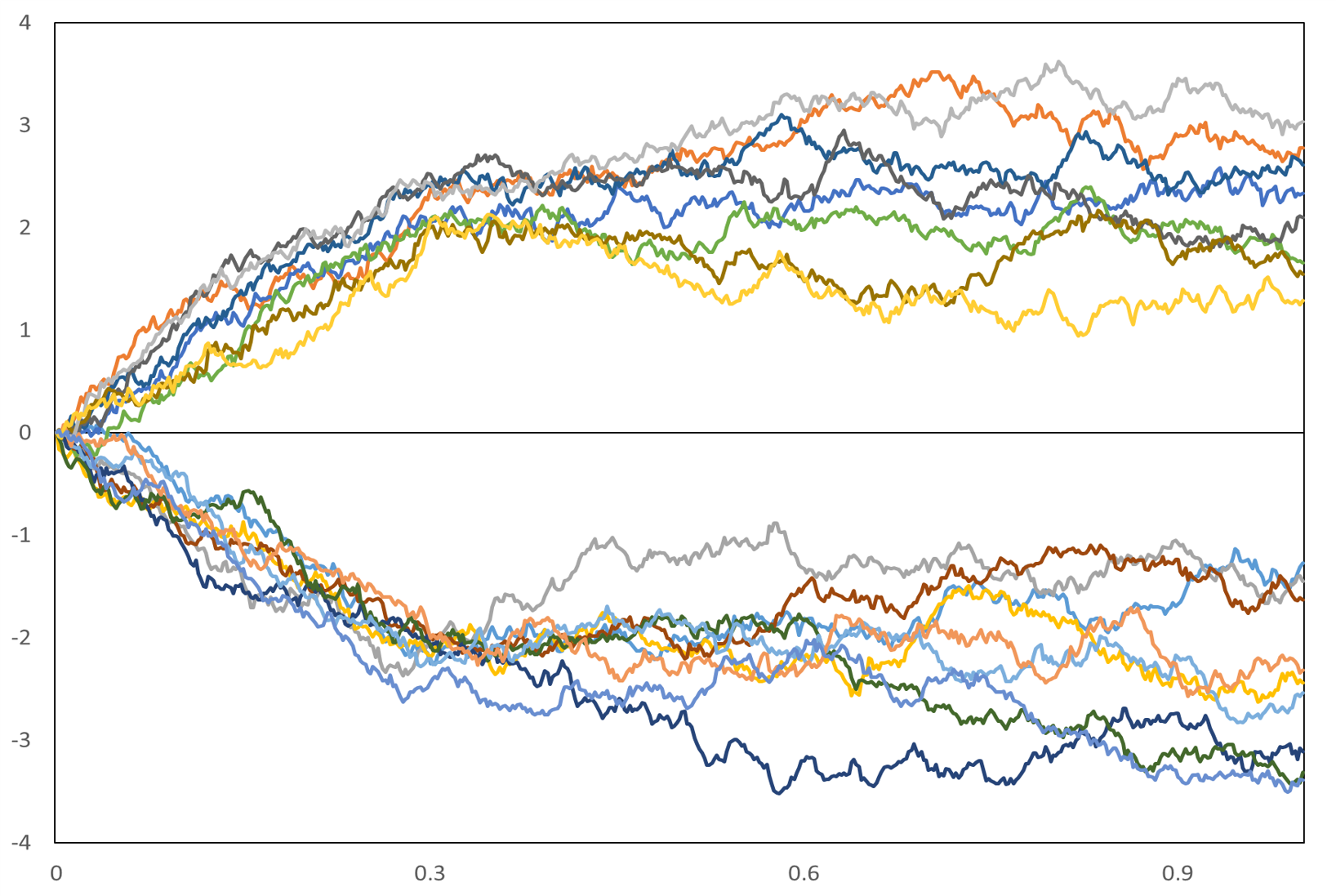}
\caption{Optimal trajectories for Example \ref{ex:3} with $T=1$, $x_0=-2$, $x_1=2$, $t_0=0.2$, obtained by standard Euler-Maruyama method with $X_0^*=0$ and time-step $\Delta t=0.005$.}
\end{figure}

\begin{remark}\label{rem:doob}
It is worth observing that the structure of the optimally controlled dynamics in \eqref{eq:X*} is formally analogous to the so-called Doob's {\em $h$-transform} (see, e.g., \cite{RW1} and \cite{RW2}). 
Fix a (smooth) domain $\Sigma\subset[0,T]\times\R^d$ and consider a Markov process $Y\in\R^d$ on a probability space $(\Omega,\cF,(\cF_t)_{t\in[0,T]},\P)$. The $h$-transform $\hat Y$ of the process $Y$ is obtained by conditioning the process $(t,Y)$ to exit $\Sigma$ via a specific portion of its boundary $\partial \Sigma_0\subset\partial \Sigma$. 

In the special case of $Y$ being solution to a SDE with drift and diffusion coefficients $\mu$ and $\sigma$, the constrained dynamics $\hat Y$ is given by 
\[
\ud \hat Y_t=\big(\mu(t,\hat Y_t)+(\sigma\sigma^\top)(t,\hat Y_t)\tfrac{\nabla h(t,\hat Y_t)}{h(t,\hat Y_t)}\big)\ud t+\sigma(t,\hat Y_t)\ud W_t,
\]
where, at least formally, 
\[
h(t,y)=\P_{t,y}\big((\tau_\Sigma,Y_{\tau_\Sigma})\in \partial \Sigma_0\big),
\] 
with $\tau_\Sigma=\inf\{s\ge t: (s,Y_s)\notin \Sigma\}$. Of course here some care is needed in order to guarantee that $h(t,y)$ is sufficiently smooth and the constrained SDE is well-posed. However, the main conceptual point is that the function $h$ should be harmonic for the process $(t,Y)$, strictly positive inside $\Sigma$ with $h|_{\partial \Sigma_0}=1$ and $h|_{\partial \Sigma\setminus\partial \Sigma_0}=0$. 
In our case $\Sigma=\cC$ and $\partial\Sigma_0=\cC_T$. We observe that our function $u$ performs an analogue task as the $h$-transform but in a stochastic control setting involving running cost and terminal cost. When $f=g\equiv 0$ we recover the $h$-transform as illustrated in the examples above.

Finally, we observe that the well-known Brownian bridge can be obtained as Doob's $h$-transform of a time-space Brownian motion, constrained to leave $\Sigma=[0,T]\times\R$ from the boundary point $\partial\Sigma_0=(T,0)$. As pointed out in \cite{FSh}, this type of process {\em cannot} be obtained as the optimal dynamics of a linear-quadratic stochastic control problem, because the drift of a Brownian bridge is not square integrable and it would yield an infinite cost in \eqref{eq:P1}. That also tallies with the fact that Assumption \ref{ass:ZD} fails.
\end{remark}

\begin{remark}\label{rem:risksensitive}
If we remove the state constraint, i.e., if we take $\cD=\varnothing$ (or more generally $\cD$ such that $\Q_{t,z}(\tau_\cD=\infty)=1$), then we fall back into the theory of logarithmic transformations of risk-sensitive functionals. In particular, the function $u$ in \eqref{eq:u} reduces to
\[
u(t,z)=\E^\Q_{t,z}\Big[\exp\Big(-\frac{1}{2}\int_t^T f(s,Z_s)\ud s-\frac{1}{2}g(Z_T)\Big)\Big]
\]
and it is well-known (see, e.g., \cite{BD} and \cite[Ch.\ VI]{FS}) that $v=-2\ln u$ is the value function of the unconstrained control problem 
\[
v(t,x)=\inf_{a\in\cA_{t,x}}\E\left[\int_t^{T}\!\!\big(f(s,X^{t,x;a}_s)+|a_s|^2_{d'}\big)\ud s+g(X^{t,x;a}_T)\right],
\]
where $\cA_{t,x}\coloneqq\cA^{\varnothing}_{t,x}$, in our notation from Definition \ref{def:admissible}. 
\end{remark}

\section{Proof of Theorem \ref{thm:main}}\label{sec:proof}

Here we provide the proof of Theorem \ref{thm:main} which is enabled by a series of results. Throughout the section we enforce Assumptions \ref{ass:ZD}, \ref{ass:L} and \ref{ass:u}. 

Recall the function 
\[u(t,z)=\E_{t,z}^\Q\Big[\exp\Big(-\frac{1}{2}\int_t^T f(s,Z_s)\ud s-\frac{1}{2}g(Z_T)\Big)1_{\{T<\tau_\cD\}}\Big]\] 
defined in \eqref{eq:u}. It is clear that $u\equiv 0$ on $\cD$, and $u(T,z)=\e^{-\frac{1}{2}g(z)}$ for $(T,z)\in\cC_T$. 

\begin{lemma}\label{lem:u}
The function $u$ satisfies $u\in C^{1,2}(\cC_{[0,T)})$, and it solves 
\begin{equation}\label{eq:PDEu}
\begin{array}{ll}
\partial_t u(t,z)+(\cL u)(t,z)-\frac{1}{2}f(t,z)u(t,z)=0,& (t,z)\in \cC_{[0,T)}.
\end{array}
\end{equation}
Moreover,
\begin{equation}\label{eq:u>0}
u(t,z)>0\iff (t,z)\in\cC.
\end{equation}
\end{lemma}

\begin{proof}
Condition \eqref{eq:u>0} follows by definition of $u$, continuity of paths of $(s,Z_s)$, Assumption \ref{ass:ZD} and Assumption \ref{ass:u}--(iii). 

Fix a point $(t,z)\in\cC_{[0,T)}$. Let $B\subset \R^d$ be an open ball and set $U=[t_1,t_2)\times B$ with $(t,z)\in U$ and $\overline U\subseteq \cC$.
Denote $\rho\coloneqq\inf\{s\geq t:(s,Z_s)\notin U\}$ the first exit time of $(s,Z_s)$ from $U$, and note that $\rho\leq \tau_\cD$ with $\tau_\cD$ as in \eqref{eq:tauD}.
Let $(\cG^Z_s)_{s\in [t,T]}$ be the filtration generated by the process $Z$, so that $Z$ is strong Markov for such filtration. 
Then $\rho$ is a $(\cG^Z_s)$-stopping time, and the strong Markov property yields
\begin{equation}\label{eq:Ymg}
\begin{aligned}
&\E^\Q_{t,z}\Big[\e^{-\frac{1}{2}\int_t^{\rho} f(r,Z_r)\ud r}u(\rho,Z_{\rho})\Big]\\
&=\E^\Q_{t,z}\Big[\e^{-\frac{1}{2}\int_t^{\rho} f(r,Z_r)\ud r}\E^\Q_{\rho,Z_{\rho}}\Big[\e^{-\frac{1}{2}\int_\rho^T f(s,Z_s)\ud s-\frac{1}{2}g(Z_T)}1_{\{T<\tau_\cD\}}\Big]\Big]\\
&=\E^\Q_{t,z}\Big[\e^{-\frac{1}{2}\int_t^{\rho} f(r,Z_r)\ud r}\E^\Q_{t,z}\Big[\e^{-\frac{1}{2}\int_\rho^T f(s,Z_s)\ud s-\frac{1}{2}g(Z_T)}1_{\{T<\tau_\cD\}}\Big|\cG^Z_\rho\Big]\Big]\\
&=\E^\Q_{t,z}\Big[\e^{-\frac{1}{2}\int_t^T f(s,Z_s)\ud s-\frac{1}{2}g(Z_T)}1_{\{T<\tau_\cD\}}\Big]=u(t,z).
\end{aligned}
\end{equation}

Since $u\in C(\overline U)$ by Assumption \ref{ass:u}, there exists a unique $w\in C(\overline U)\cap C^{1,2}(U)$ that solves
\begin{equation}
\left\{\begin{array}{ll}
\partial_t w(t,z)+(\cL w)(t,z)-\frac{1}{2}f(t,z)w(t,z)=0,& (t,z)\in U,\\[+5pt]
w(t,z)=u(t,z),& (t,z)\in \partial_P U,
\end{array}\right.
\end{equation}
thanks to Assumption \ref{ass:L}.
By Dynkin's formula, for any $(t,z)\in U$ we have 
\begin{equation}
w(t,z)=\E^\Q_{t,z}\Big[\e^{-\frac{1}{2}\int_t^{\rho} f(r,Z_r)\ud r}w(\rho,Z_{\rho})\Big]
=\E^\Q_{t,z}\Big[\e^{-\frac{1}{2}\int_t^{\rho} f(r,Z_r)\ud r}u(\rho,Z_{\rho})\Big]=u(t,z),
\end{equation}
where the second equality holds because $(\rho,Z_\rho)\in\partial_P U$ and the final equality holds by \eqref{eq:Ymg}.
Since $U$ is arbitrary, \eqref{eq:PDEu} holds. 
\end{proof}

Let us introduce the {\em Hamiltonian} $\cH:[0,T]\times\R^d\times\R^d\times\R^{d\times d}\to\R$, defined as
\[\cH(t,x,p,\gamma)\coloneqq \tfrac{1}{2}\mathrm{tr}(\sigma\sigma^\top \gamma)(t,x)+\langle \mu,p \rangle(t,x)
+\inf_{\alpha\in\R^{d'}}\big(\langle \sigma\alpha,p\rangle+|\alpha|^2_{d'}\big)(t,x).
\]
Now we have a simple corollary of the lemma above.
\begin{corollary}\label{cor:V}
Set $V(t,x)\coloneqq-2\ln u(t,x)$ for $(t,x)\in \cO$. Then $|V(t,x)|<\infty$ for $(t,x)\in\cC$ and $V(t,x)=+\infty$ for $(t,x)\in\cD$. Moreover, $V\in C^{1,2}(\cC_{[0,T)})\cap C(\overline{\cC}_{[0,T)}\cup\cC_T)$ and it satisfies the Hamilton-Jacobi-Bellman equation
\begin{equation}\label{eq:HJB}
\left\{
\begin{array}{ll}
\partial_t V(t,x)+\cH(t,x,\nabla V(t,x), D^2 V(t,x)) + f(t,x)=0, & (t,x)\in\cC_{[0,T)},\\[+5pt]
V(T,x)=g(x), & (T,x)\in\cC_T.
\end{array}
\right.
\end{equation}
\end{corollary}

\begin{proof}
Finiteness of $V$ follows from \eqref{eq:u>0} and its continuity/smoothness in $\overline{\cC}_{[0,T)}\cup\cC_T$ from that of $u$. Moreover, $V=+\infty$ in $\cD$, because $u=0$ in $\cD$. Differentiating $V$ and using \eqref{eq:PDEu} we obtain
\begin{equation}\label{eq:PDEV}
\partial_t V(t,x)+(\cL V)(t,x)-\tfrac{1}{4}\big|\sigma^\top(t,x)\nabla V(t,x) \big|^2_d+f(t,x)=0,
\end{equation}
for every $(t,x)\in \cC_{[0,T)}$.
Then, \eqref{eq:HJB} is obtained upon noticing that 
\begin{equation}\label{eq:Ham}
-\tfrac{1}{4}\big|\sigma^\top\nabla V \big|^2(t,x)=\inf_{\alpha\in\R^{d'}}\big(\langle \sigma\alpha,\nabla V\rangle+|\alpha|^2_{d'}\big)(t,x).
\end{equation}
The expression for $V(T,x)$ holds by definition of $u(T,x)$. 
\end{proof}

Another consequence of Lemma \ref{lem:u} is the existence of a (local) strong solution of the SDE in \eqref{eq:X*}, that rules the optimal dynamics in Theorem \ref{thm:main}. 
Recall that $\cD=\{(t,x):u(t,x)=0\}$, and for $\eps >0$ let us introduce the sets 
\[
 \ \cC_\eps\coloneqq\{(t,x):u(t,x)>\eps\}\quad\text{and}\quad \cD_\eps\coloneqq \cO\setminus\cC_\eps.
\]
For sufficiently small $\eps>0$ the set $\cC_{\eps}$ is non-empty and contained in $\cC$. 
Since  $\cD\subseteq\cD_{\eps'}\subseteq\cD_{\eps}$ for any $\eps'\le \eps$, the limit $\cD_{0}\coloneqq\lim_{\eps\downarrow 0}\cD_{\eps}$ is well-defined and it is easy to see that $\cD_{0}= \cD$. 
Also, for $\eps,\delta>0$ we set 
\[\cC_{\eps,\delta}\coloneqq\{(t,x)\in\cO:u(t,x)>\eps\}\cap \{(t,x)\in\cO:\vert x\vert< \delta^{-1}\}.\]
Thanks to Assumption \ref{ass:u}--(i), for any $\eps, \delta, \eta>0$ 
\begin{equation}\label{eq:distance}
\inf_{\substack{(t,x)\in\cC_{\eps,\delta}\\ 0\le t\le T-\eta}}\mathrm{dist}\big((t,x),\cD\big)>0.
\end{equation}

\begin{proposition}\label{prop:strong}
Fix $(t,x)\in \cC$ and recall $\alpha^*$ as in \eqref{eq:a*}. For any $\eps,\delta,\eta>0$ there exists a unique strong solution $X^*$ of 
\begin{equation}\label{eq:X*locEE}
\begin{aligned}
X^*_{s\wedge\rho^*_{\eps,\delta,\eta}}&=x\!+\!\int_t^{s\wedge\rho^*_{\eps,\delta,\eta}}\!\big[\mu(u,X^*_u)\!+\!\sigma(u,X^*_u)\alpha^*(u,X^*_u)\big]\ud u\!+\!\int_t^{s\wedge\rho^*_{\eps,\delta,\eta}}\!\! \sigma(u,X^*_u)\ud W_u, 
\end{aligned}
\end{equation}
for $s\in[t,T]$, where 
\[\rho^*_{\eps,\delta,\eta}\coloneqq \inf\{r\in[t,T-\eta]:(r,X^*_r)\notin \cC_{\eps,\delta}\}\wedge (T-\eta),\]
with the usual convention $\inf\varnothing=\infty$.
\end{proposition}
\begin{proof}
For any $\eps,\delta,\eta>0$ the functions $\mu(t,x)$, $\sigma(t,x)$ and  $\alpha^*(t,x)=\sigma^\top(t,x)(\nabla u/u)(t,x)$ are Lipschitz on $\cC_{\eps,\delta}\cap ([0,T-\eta]\times \R^d)$, thanks to \eqref{eq:distance} and Lemma \ref{lem:u}.
Then the result follows by classical theory for SDEs.
\end{proof}

For fixed $\eps,\delta>0$, the time $\rho^*_{\eps,\delta,\eta}$ increases as $\eta\downarrow 0$, and the solution $X^*$ can be extended so that 
\begin{equation}\label{eq:X*loc}
\begin{aligned}
X^*_{s\wedge\rho^*_{\eps,\delta}}&=x\!+\!\int_t^{s\wedge\rho^*_{\eps,\delta}}\!\big[\mu(u,X^*_u)\!+\!\sigma(u,X^*_u)\alpha^*(u,X^*_u)\big]\ud u\!+\!\int_t^{s\wedge\rho^*_{\eps,\delta}}\!\! \sigma(u,X^*_u)\ud W_u
\end{aligned}
\end{equation}
for $s\in[t,T)$, where 
\begin{equation}\label{eq:sigma**}
\rho^*_{\eps,\delta}\coloneqq \lim_{\eta\downarrow 0}\rho^*_{\eps,\delta,\eta} = \inf\{r\in[t,T):(r,X^*_r)\notin \cC_{\eps,\delta}\}\wedge T.
\end{equation}

We next provide a pathwise uniqueness result 
on the closed time interval $[t,T]$. 

\begin{proposition}\label{prop:pathwise}
Let $\alpha:\cC_{[0,T)}\to\R^{d'}$ be measurable and locally Lipschitz in the spatial variable. For $(t,x)\in\cC$, there can be at most one strong solution of the SDE
\begin{equation}\label{eq:Xaa}
X_s=x+\int_t^s \big[\mu(u,X_u)+\sigma(u,X_u)\alpha(u,X_u)\big]\ud u+\int_t^s \sigma(u,X_u)\ud W_u,\quad s\in[t,T],
\end{equation}
with 
\begin{equation}\label{eq:nonexpl}
\P\big((s,X_s)\in\cC,\,\forall s\in[t,T]\big)=1.
\end{equation} 
\end{proposition}

\begin{proof}
In order to show that pathwise uniqueness holds let us assume that $X$ and $X'$ be two solutions of \eqref{eq:Xaa} (with the same Brownian motion) satisfying \eqref{eq:nonexpl}. 
For $\eps,\delta>0$ we introduce
\[
\rho_{\eps,\delta}\coloneqq\inf\{r\in[t,T]:(r,X_r)\notin \cC_{\eps,\delta}\text{ or }(r,X'_r)\notin \cC_{\eps,\delta}\}\wedge T,
\]
with the usual convention $\inf\varnothing=\infty$. Notice that for $\eps'<\eps$ and $\delta'<\delta$ we have $\rho_{\eps',\delta'}\ge \rho_{\eps,\delta}$.
By conditions \eqref{eq:nonexpl} we have $\rho_{\eps,\delta}\uparrow T$
as $\eps,\delta\to 0$, $\P$-a.s. Letting also $0<\eta<T-t$ we set $\rho_{\eps,\delta,\eta}\coloneqq \rho_{\eps,\delta}\wedge(T-\eta)$.
By Fatou's lemma, for any $s\in[t,T]$ we have
\begin{equation}\label{eq:fatou}
\E\Big[\big|X_s-X'_s\big|^2_d\Big]\le \liminf_{\eta,\eps,\delta\to 0}\E\Big[\big|X_{s\wedge\rho_{\eps,\delta,\eta}}-X'_{s\wedge\rho_{\eps,\delta,\eta}}\big|^2_d\Big].
\end{equation}
For the ease of notation we set $\bar \alpha(t,x)\coloneqq \mu(t,x)+\sigma(t,x)\alpha(t,x)$. Using It\^o's formula and the fact that the coefficients in \eqref{eq:Xaa} are Lipschitz on $\cC_{\eps,\delta}\cap ([t,T-\eta]\times \R^d)$
with constant $L=L_{\eta,\eps,\delta}>0$ we immediately obtain
\begin{equation}
\begin{aligned}
&\E\Big[\big|X_{s\wedge\rho_{\eps,\delta,\eta}}-X'_{s\wedge\rho_{\eps,\delta,\eta}}\big|^2_d\Big]\\
&=2\E\Big[\int_t^{s\wedge\rho_{\eps,\delta,\eta}}\Big(\big\langle X_r-X'_r,\bar \alpha(r,X_r)-\bar \alpha(r,X'_r)\big\rangle+\tfrac12\big|\sigma(r,X_r)-\sigma(r,X'_r)\big|^2_{d\times d'}\Big)\ud r\Big]\\
&\le 3L\E\Big[\int_t^{s}\big| X_{r\wedge\rho_{\eps,\delta,\eta}}-X'_{r\wedge\rho_{\eps,\delta,\eta}} \big|^2_d\ud r\Big].
\end{aligned}
\end{equation}
Gronwall's inequality yields $\E[|X_{s\wedge\rho_{\eps,\delta,\eta}}-X'_{s\wedge\rho_{\eps,\delta,\eta}}|^2_d]=0$ which, combined with \eqref{eq:fatou}, implies that $X$ and $X'$ are modifications. Since the processes are continuous, they are indistinguishable on $[t,T]$. That is, the solution of \eqref{eq:Xaa} is pathwise unique.
\end{proof}

The plan is now to extend the strong solution in \eqref{eq:X*loc} to $[t,T]$ and let $\eps, \delta\to 0$.  
We begin with an estimate of the candidate optimal control.

\begin{lemma}
Fix $(t,x)\in\cC$. For $\alpha^*$ as in \eqref{eq:a*} we have
\begin{equation}\label{eq:intAd}
\E\Big[\sup_{\eps,\delta>0}\int_{t}^{\rho^*_{\eps,\delta}}|\alpha^*(r,X^*_r)|^2\ud r\Big]=\sup_{\eps,\delta>0}\sup_{s\in[t,T)}\E\Big[\int_{t}^{s\wedge\rho^*_{\eps,\delta}}|\alpha^*(r,X^*_r)|^2\ud r\Big]<\infty.
\end{equation}
\end{lemma}

\begin{proof}
The equality holds by Monotone Convergence theorem, because for $\delta'<\delta$ and $\eps'<\eps$ we have $\rho^*_{\eps',\delta'}\ge \rho^*_{\eps,\delta}$ and therefore the supremum is attained as limit along any decreasing sequence $(\delta_n,\eps_n)\to 0$. It remains to show that the right-hand side of \eqref{eq:intAd} is finite.

Recall $V\coloneqq-2\ln u$ from Corollary \ref{cor:V}. 
Fix arbitrary $\eps,\delta>0$ and take $(t,x)\in\cC_{\eps,\delta}$ with $t<T$. 
Applying Dynkin's formula and using the definition of the control $\alpha^*(r,X^*_r)$ from \eqref{eq:a*}, we obtain for $s\in[t,T)$
\begin{equation}
\begin{aligned}
V(t,x)&=\E\Big[V(s\wedge\rho^*_{\eps,\delta},X^{*}_{s\wedge\rho^*_{\eps,\delta}})\!-\!\int_{t}^{s\wedge\rho^*_{\eps,\delta}}\!\!\big(\partial_t V\!+\!\cL V \!+\!\langle\nabla V, \sigma \alpha^*\rangle\big)(r,X^{*}_r) \ud r\Big]\\
&=\E\Big[V(s\wedge\rho^*_{\eps,\delta},X^*_{s\wedge\rho^*_{\eps,\delta}})\!-\!\int_{t}^{s\wedge\rho^*_{\eps,\delta}}\!\!\Big(\!\big(\partial_t V\!+\!\cH(\cdot,\nabla V, D^2 V)\big)(r,X^*_r)\!-\!|\alpha^*(r,X^*_r)|^2\Big)\ud r\Big]\\
&= \E\Big[ 
V(s\wedge\rho^*_{\eps,\delta},X^*_{s\wedge\rho^*_{\eps,\delta}})\!+\!\int_{t}^{s\wedge\rho^*_{\eps,\delta}}\!\Big(f(r,X^{*}_r)\!+\!|\alpha^*(r,X^*_r)|^2\Big)\ud r\Big],
\end{aligned}
\end{equation}
where we used \eqref{eq:PDEV} and \eqref{eq:Ham} for the second and third equality. Since $V,f\ge 0$ we obtain 
\begin{equation}
\sup_{\eps,\delta>0}\sup_{s\in[t,T)}\E\Big[\int_{t}^{s\wedge\rho^*_{\eps,\delta}}|\alpha^*(r,X^*_r)|^2\ud r\Big]\le V(t,x)<\infty.
\end{equation}
Hence, \eqref{eq:intAd} holds.
\end{proof}

Next we establish that the process $X^*$ is non-explosive. We cannot rely on classical stability estimates for SDEs because controlling the term $\sigma(u,X^*_u)\alpha^*(u,X^*_u)$ in an $L^p$-estimate, under the sole assumption of linear growth for $\sigma$, would require a-priori knowledge of $L^{p'}$-bounds for both factors in the product, with $p'>p$. Here instead we prove that the solution does not explode using the integrability property \eqref{eq:intAd}.

\begin{proposition}\label{prop:blowup}
Fix $(t,x)\in \cC$ and recall $\alpha^*$ as in \eqref{eq:a*}. 
Then
\begin{equation}\label{eq:Pfinite}
\sup_{\eps,\delta>0}\sup_{s\in[t,T)}\big|X^*_{s\wedge\rho^*_{\eps,\delta}}\big|<\infty,\quad\P-a.s.
\end{equation}
\end{proposition}

\begin{proof}
Take $\psi(z)\coloneqq\ln(1+z)$ for $z\ge 0$. An application of It\^o's formula yields, for $s\in[t,T)$ 
\begin{equation}
\begin{aligned}
\psi\big(\big|X^*_{s\wedge\rho^*_{\eps,\delta}}\big|_d^2\big)&\le\psi\big(|x|^2_d\big)+\int_t^{s\wedge\rho^*_{\eps,\delta}}\frac{2\langle X^*_r,\mu(r,X^*_r)+(\sigma \alpha^*)(r,X^*_r)\rangle+|\sigma(r,X^*_r)|^2_{d\times d'}}{1+|X^*_r|^2_d}\ud r\\
&\quad+2\int_t^{s\wedge\rho^*_{\eps,\delta}}\frac{1}{1+|X^*_r|^2_d}\langle X^{*}_r,\sigma(r,X^*_r)\ud W_r\rangle,
\end{aligned}
\end{equation}
where the inequality holds by concavity of $\psi$. The linear-growth assumptions on $\mu$ and $\sigma$ guarantee that there is a constant $c_0>0$, independent of $\eps$ and $\delta$ such that 
\begin{equation}\label{eq:linbound}
\begin{aligned}
&\Big|\frac{\langle x,\mu(r,x)\rangle}{1+|x|^2_d}\Big|+\frac{|\sigma(r,x)|^2_{d\times d'}}{1+|x|^2_d}+\frac{|\sigma^\top(t,x)x|_{d'}}{1+|x|^2_d}\le c_0,\\
&\Big|\frac{\langle x,\sigma(r,x) \alpha^*(r,x)\rangle}{1+|x|^2_d}\Big|\le c_0 \big(1+\big|\alpha^*(r,x)\big|_d\big).
\end{aligned}
\end{equation}
Set 
\[
N_{s\wedge\rho^*_{\eps,\delta}}\coloneqq 2\int_t^{s\wedge\rho^*_{\eps,\delta}}\frac{1}{1+|X^*_r|^2_d}\langle X^{*}_r,\sigma(r,X^*_r)\ud W_r\rangle
\]
to simplify notation, and observe that $s\mapsto N_{s\wedge\rho^*_{\eps,\delta}}$ is a continuous, uniformly integrable martingale on $[t,T)$, hence a martingale
on $[t,T]$. Jensen's inequality and Doob's inequality yield
\begin{equation}
\begin{aligned}
\E\Big[\sup_{s\in[t,T]}\big|N_{s\wedge\rho^*_{\eps,\delta}}\big|\Big]&\le \E\Big[\sup_{s\in[t,T]}\big|N_{s\wedge\rho^*_{\eps,\delta}}\big|^2\Big]^\frac12 \le 2\E\Big[\big|N_{\rho^*_{\eps,\delta}\wedge T}\big|^2\Big]^\frac12\\
&=4\E\Big[\int_t^{\rho^*_{\eps,\delta}\wedge T}\frac{\big|\sigma^\top(r,X^*_r)X^*_r\big|^2_{d'}}{(1+|X^*_r|^2_d)^2}\ud r\Big]^\frac12\le 4c_0\sqrt{T},
\end{aligned}
\end{equation}
where the final inequality holds by \eqref{eq:linbound}.

Since $\psi$ is a monotonic function, taking supremum over times and expectation, and using the above bounds, imply that
\begin{equation}
\begin{aligned}
&\E\Big[\psi\Big(\sup_{ s\in[t, T)}|X^*_{s\wedge \rho^*_{\eps,\delta}}|^2_d\Big)\Big]=\E\Big[\sup_{ s\in[t, T)}\psi\big(|X^*_{s\wedge \rho^*_{\eps,\delta}}|^2_d\big)\Big]\\
&\le \psi(|x|^2_d)+c_0\Big(2 T+4\sqrt{T}+\E\Big[\int_t^{\rho^*_{\eps,\delta}}\big|\alpha^*(r,X^*_r)\big|_d\ud r\Big]\Big)\le \psi(|x|^2_d)+ C_1,
\end{aligned}
\end{equation}  
where for the final inequality we use \eqref{eq:intAd} and $C_1>0$ is a suitable constant independent of $\eps,\delta$.  
Since $\rho^*_{\eps',\delta'}\ge \rho^*_{\eps,\delta}$ for $\delta'<\delta$ and $\eps'<\eps$, we have
\[
\sup_{s\in[t,T)}\big|X^*_{s\wedge\rho^*_{\eps,\delta}}\big|\le \sup_{s\in[t,T)}\big|X^*_{s\wedge\rho^*_{\eps',\delta'}}\big|,\quad \P-a.s.
\]
Therefore, by monotonicity of $\psi$ and using the Monotone Convergence theorem, 
\begin{equation}\label{eq:psi1}
\begin{aligned}
\E\Big[\psi\Big(\sup_{\eps,\delta>0}\sup_{s\in[t, T)}|X^*_{s\wedge \rho^*_{\eps,\delta}}|^2_d\Big)\Big]\le \psi(|x|^2_d)+ C_1,
\end{aligned}
\end{equation}
and since $\psi$ is a strictly increasing function with $\lim_{z\to\infty}\psi(z)=+\infty$ we then conclude that 
\[
\sup_{\eps,\delta>0}\sup_{s\in[t,T)}\vert X^*_{s\wedge \rho^*_{\eps,\delta}}|_d<\infty,\quad\P-a.s.
\] 
as needed.
\end{proof}

Propositions \ref{prop:strong} and \ref{prop:blowup} allow us to let $\delta\to 0$ and define 
\begin{equation}\label{eq:rhoeps}
\rho^*_{\eps}\coloneqq\lim_{\delta\to 0}\rho^*_{\eps,\delta}= \inf\{r\in[t,T):(r,X^*_r)\notin \cC_{\eps}\}\wedge T.
\end{equation}
Moreover, we have the next result.

\begin{corollary}\label{cor:X*loc}
Fix $(t,x)\in \cC$. For any $\eps>0$ there exists a unique strong solution $X^*$ of 
\begin{equation}\label{eq:X*loc3}
\begin{aligned}
X^*_{s\wedge\rho^*_{{\eps}}}&=x\!+\!\int_t^{s\wedge\rho^*_{{\eps}}}\!\!\big[\mu(u,X^*_u)\!+\!\sigma(u,X^*_u)\alpha^*(u,X^*_u)\big]\ud u\!+\!\int_t^{s\wedge\rho^*_{{\eps}}}\!\! \sigma(u,X^*_u)\ud W_u, 
\end{aligned}
\end{equation}
for $s\in[t,T)$. Moreover, 
\begin{equation}\label{eq:intAd2}
\E\Big[\sup_{\eps>0}\int_{t}^{\rho^*_{\eps}}|\alpha^*(r,X^*_r)|^2\ud r\Big]=\sup_{\eps>0}\sup_{s\in[t,T)}\E\Big[\int_{t}^{s\wedge\rho^*_{\eps}}|\alpha^*(r,X^*_r)|^2\ud r\Big]<\infty.
\end{equation}
\end{corollary}
Notice that, as in \eqref{eq:intAd}, the first equality in \eqref{eq:intAd2} holds by monotone convergence because $\rho^*_{\eps'}\ge \rho^*_\eps$ for $\eps'<\eps$ and therefore the supremum is attained along any decreasing sequence $\eps_n\to 0$. The inequality in \eqref{eq:intAd2} follows by the same arguments as in the proof of the inequality in \eqref{eq:intAd}.
We now proceed to construct a  strong solution of \eqref{eq:X*} on $[t,T]$. 

\begin{proposition}\label{prop:weakX}
Fix $(t,x)\in \cC$. Assume that
\begin{equation}\label{eq:intA}
\begin{aligned}
\quad\lim_{\eps\to 0}\P(\rho^*_{{\eps}}< T)=0.
\end{aligned}
\end{equation}
Then, there exists a unique strong solution of 
\begin{equation}\label{eq:X*2}
\begin{aligned}
X^*_{s}&=x+\int_t^{s}\big[\mu(u,X^*_u)+\sigma(u,X^*_u)\alpha^*(u,X^*_u)\big]\ud u+\int_t^{s} \sigma(u,X^*_u)\ud W_u,
\end{aligned}
\end{equation}
for $s\in[t,T)$, with $(s,X^*_s)\in\cC$ for all $s\in[t,T)$, $\P_{t,x}$-a.s. Moreover, the limit $X^*_T\coloneqq\lim_{s\uparrow T}X^*_s$ is well-defined.
\end{proposition}

\begin{remark}
The proposition does not yet ensure that $a^*_s\coloneqq \alpha^*(s,X^*_s)$ is an admissible control, because it may be $(T,X^*_T)\in\cD_T$. This will be ruled out later, in Proposition \ref{prop:limXT}.
\end{remark}

\begin{proof}
Let $(\eps_n)_{n\in\N}\subset(0,1)$ be such that $\eps_n\downarrow 0$ as $n\to\infty$. Then $\rho^*_{\eps_n}\le\rho^*_{\eps_{n+1}}$ and the limit $\rho^*_0\coloneqq\lim_{n\to\infty}\rho^*_{\eps_n}$ is well-defined, $\P$-a.s. Moreover, $\{\rho^*_{\eps_n}< T\}\supset \{\rho^*_{\eps_{n+1}}< T\}$ and $\lim_{n\to\infty}\{\rho^*_{\eps_n}< T\}=\cap_{n\in\N}\{\rho^*_{\eps_n}< T\}$.
By  \eqref{eq:intA} we have 
\begin{equation}\label{eq:avoid}
\P(\cup_{n\in\N}\{\rho^*_{\eps_n}\ge T\})=\lim_{n\to\infty}\P(\rho^*_{\eps_n}\ge T)=1,
\end{equation}
or, equivalently, $\P(\rho^*_0\ge T)=1$. 
In particular, for any $s\in[t,T)$ and any small $\gamma>0$, there is $N_{\gamma,s}\in\N$ such that $\P(\rho^*_{\eps_n}>s)\ge 1-\gamma$ for all $n\ge N_{\gamma,s}$. Setting 
\[
\Omega_{\gamma,s}\coloneqq\cap_{n\ge N_{\gamma,s}}\{\rho^*_{\eps_n}>s\}
\]
and noticing that $\{\rho^*_{\eps_n}>s\}\subset\{\rho^*_{\eps_{n+1}}>s\}$ we have $\P(\Omega_{\gamma,s})\ge 1-\gamma$. Taking $\eps=\eps_n$ in \eqref{eq:X*loc3}, with $n\ge N_{\gamma,s}$, we have on the event $\Omega_{\gamma,s}$
\begin{equation}\label{eq:X*loc2}
\begin{aligned}
X^*_{r}&=x+\int_t^{r}\big[\mu(u,X^*_u)+\sigma(u,X^*_u)\alpha^*(u,X^*_u)\big]\ud u+\int_t^{r} \sigma(u,X^*_u)\ud W_u, 
\end{aligned}
\end{equation}
for all $r\in[t,s]$. Moreover, $(r,X^*_r)\in\cC$ for all $r\in[t,s]$ on $\Omega_{\gamma,s}$, by \eqref{eq:distance} (due to Assumption \ref{ass:u}-(i)). 

Taking a sequence $(\gamma_k)_{k\in\N}\subset(0,1)$ with $\gamma_k\downarrow 0$
and setting $\Omega_s\coloneqq\cup_{k\in\N}\Omega_{\gamma_k,s}$ we have $\P(\Omega_s)=1$. On the event $\Omega_s$ the process $X^*$ is the unique strong solution of \eqref{eq:X*loc2} on $[t,s]$ with 
$(r,X^*_{r})\in\cC$, $\forall r\in[t,s]$. Notice that uniqueness may be deduced by the same argument as in the proof of Proposition \ref{prop:pathwise}. Finally, taking a sequence $(s_j)_{j\in\N}\subset[t,T)$ with $s_j\uparrow T$ and setting $\Omega_T\coloneqq\cap_{j\in\N}\Omega_{s_j}$ we have $\P(\Omega_T)=1$. Then, on the event $\Omega_T$ the process $X^*$ is the unique strong solution of \eqref{eq:X*loc2} on $[t,T)$ with 
$(r,X^*_{r})\in\cC$, $\forall r\in[t,T)$.

From \eqref{eq:intAd2} we deduce
\begin{equation}
\begin{aligned}
\infty&>\liminf_{j\to\infty}\liminf_{n\to\infty}\E\Big[\int_t^{s_j\wedge\rho^*_{\eps_n}}\big|\alpha^*(r,X^*_r)\big|^2\ud r\Big]\\
&\ge \liminf_{j\to\infty}\E\Big[1_{\Omega_T}\liminf_{n\to\infty}\int_t^{s_j\wedge\rho^*_{\eps_n}}\big|\alpha^*(r,X^*_r)\big|^2\ud r\Big]\\
&=\liminf_{j\to\infty}\E\Big[1_{\Omega_T}\int_t^{s_j}\big|\alpha^*(r,X^*_r)\big|^2\ud r\Big]\ge \E\Big[1_{\Omega_T}\int_t^{T}\big|\alpha^*(r,X^*_r)\big|^2\ud r\Big],
\end{aligned}
\end{equation}
where we used Fatou's lemma in the second and in the last inequality. Let $\Omega_0$ with $\P(\Omega_0)=1$ be the set where \eqref{eq:Pfinite} and $\int_t^{T}\big|\alpha^*(r,X^*_r)\big|^2\ud r<\infty$ hold. 
Continuity of $\mu$ and $\sigma$ yield for $\omega\in\Omega_0$
\begin{equation}\label{eq:ext0}
\begin{aligned}
&\Big(\int_t^T\big|\mu\big(r,X^*_r(\omega)\big)\big|_d\ud r
+\int_t^T\big|\sigma(r,X^*_r(\omega)\big)\alpha^*\big(r,X^*_r(\omega)\big)\big|_d\ud r\Big)\\
&\le T \sup_{s\in[t,T)}\big|\mu\big(s,X^*_s(\omega)\big)\big|+\sup_{s\in[t,T)}\big|\sigma\big(s,X^*_s(\omega)\big)\big|_{d\times d'}\int_t^T\big|\alpha^*\big(r,X^*_r(\omega)\big)\big|_{d'}\ud r<\infty.
\end{aligned}
\end{equation}
Likewise
\begin{equation}\label{eq:ext1}
\begin{aligned}
\int_t^T\big|\sigma\big(r,X^*_r(\omega)\big)\big|^2_{d\times d'}\ud r\le T\sup_{s\in[t,T)}\big|\sigma\big(s,X^*_s(\omega)\big)\big|^2_{d\times d'}<\infty.
\end{aligned}
\end{equation}
Combining \eqref{eq:ext0} and \eqref{eq:ext1} we deduce that the limit $X^*_T(\omega)=\lim_{s\uparrow T}X^*_s(\omega)$ is well-defined
for $\P$-a.e.\ $\omega\in\Omega$. That is, we have extended the unique strong solution of \eqref{eq:X*loc2} to the closed interval $[t,T]$, $\P$-a.s.

Invoking Proposition \ref{prop:pathwise} we conclude that there exists a unique strong solution of \eqref{eq:X*2} with $(s,X^*_s)\in\cC$ for all $s\in[t,T)$, $\P$-a.s.
\end{proof}

It remains to prove that the solution $X^*$ constructed above is actually admissible, in the sense that also $\P((T,X^*_T)\in\cC_T)=1$. For that we need Assumption \ref{ass:u}-(ii) and a well-known link between the laws of the processes $X^*$ and $Z$ (cf.\ \cite{DP}). 
Thanks to Lemma \ref{lem:u} we know that the process
\begin{equation}\label{eq:Theta}
\Theta^\cD_{s}\coloneqq\e^{-\frac12\int_t^{s\wedge\tau_\cD}f(r,Z_r)\ud r}u(s\wedge\tau_\cD,Z_{s\wedge\tau_\cD}), \quad s\in[t,T),
\end{equation}
is a bounded, non-negative $(\cG^Z_s,\Q)$-martingale. Then the limit $\Theta^\cD_T\coloneqq\lim_{s\to T}\Theta^\cD_s$ exists and $\Theta^\cD_s$ is a martingale for $s\in[t,T]$ (cf.\ \cite[Thm.\ 3.15+Prob.\ 3.20]{KS}). Using items (i), (ii) and (iii) in Assumption \ref{ass:u}, we have 
\[
\Theta_T^\cD=\e^{-\frac12\int_t^{T}f(r,Z_r)\ud r-\frac12 g(Z_T)}1_{\{T<\tau_\cD\}},\quad \Q-a.s.
\]
Then $\E^\Q_{t,x}[\Theta^\cD_{T}]=u(t,x)$. We define a probability measure $\M$ on $\cG^Z_T$ via the Radon-Nikodym derivative
\begin{equation}\label{eq:MQ}
\frac{\ud \M}{\ud \Q}\Big|_{\cG^Z_T}\coloneqq\frac{\Theta^\cD_{T}}{u(t,x)}=\frac{1}{u(t,x)}\e^{-\frac12\int_t^{T}f(r,Z_r)\ud r-\frac12 g(Z_T)}1_{\{T<\tau_\cD\}}.
\end{equation}
Clearly $\M$ is absolutely continuous with respect to $\Q$ and we will use it to establish the next lemma. For $(t,x)\in\cC$ let us introduce the stopping time 
\[
\rho^*\coloneqq\inf\{s\in[t,T]:(s,X^*_s)\notin\cC\}\wedge T,
\]
with $\inf\varnothing =+\infty$.

\begin{lemma}\label{lem:laws}
Let \eqref{eq:intA} hold. For any $(t,x)\in\cC$
the law of the process $(Z_{s\wedge\tau_\cD})_{s\in[0,T)}$ under the measure $\M_{t,x}$ is the same as the law of the process $(X^*_{s\wedge\rho^*})_{s\in[0,T)}$ under the measure $\P_{t,x}$.
\end{lemma}

\begin{proof}
Fix $(t,x)\in\cC_\eps$ and let us define
$\tau_{\cD_\eps}\coloneqq\inf\{r\in[t,T]:(r,Z_r)\notin \cC_\eps\}$ with $\inf\varnothing=+\infty$.
Letting $\E^\M_{t,x}$ be the expectation under the measure $\M_{t,x}$ and $\psi\in C^\infty_c(\R^d)$ an arbitrary test function, we have for any 
$t\le s< T$,
\begin{equation}
\begin{aligned}
&\E^\M_{t,x}\Big[\psi(Z_{s\wedge\tau_{\cD_\eps}})-\psi(x)\Big]\\
&=\frac{1}{u(t,x)}\E^\Q_{t,x}\Big[\Theta^\cD_{s\wedge\tau_{\cD_\eps}}\psi(Z_{s\wedge\tau_{\cD_\eps}})-u(t,x)\psi(x)\Big]\\
&=\frac{1}{u(t,x)}\E^\Q_{t,x}\Big[\int_t^{s\wedge\tau_{\cD_\eps}}\!\!\Big(\Theta^\cD_r (\cL\psi)(r,Z_r)+\e^{-\frac12\int_t^r f(v,Z_v)\ud v}\langle\sigma\sigma^\top\nabla u(r,Z_r),\nabla\psi(Z_r)\rangle\Big)\ud r\Big]\\
&=\frac{1}{u(t,x)}\E^\Q_{t,x}\Big[\int_t^{s\wedge\tau_{\cD_\eps}}\!\!\Theta^\cD_r\Big( (\cL\psi)(r,Z_r)+\langle\sigma\sigma^\top\nabla\ln u(r,Z_r),\nabla\psi(Z_r)\rangle\Big)\ud r\Big]\\
&=\frac{1}{u(t,x)}\E^\Q_{t,x}\Big[\Theta^\cD_{T\wedge\tau_{\cD_\eps}}\int_t^{s\wedge\tau_{\cD_\eps}}\!\!\Big( (\cL\psi)(r,Z_r)+\langle\sigma\sigma^\top\nabla\ln u(r,Z_r),\nabla\psi(Z_r)\rangle\Big)\ud r\Big]\\
&=\E^\M_{t,x}\Big[\int_t^{s\wedge\tau_{\cD_\eps}}\!\!\Big( (\cL\psi)(r,Z_r)+\langle\sigma\sigma^\top\nabla\ln u(r,Z_r),\nabla\psi(Z_r)\rangle\Big)\ud r\Big],
\end{aligned}
\end{equation}
where we can apply It\^o's formula in the second equality because $u\in C^{1,2}(\cC_{[0,T)})$ and in the third equality we used \eqref{eq:Theta}. The fourth equality holds by the martingale property of $\Theta^\cD_r$. Recalling the definition of $\alpha^*(t,x)=\sigma^\top(t,x) \nabla\ln u(t,x)$ for $(t,x)\in\cC_{[0,T)}$, the equation above can be rewritten as 
\begin{equation}\label{eq:X*b}
\E^\M_{t,x}\Big[\psi(Z_{s\wedge\tau_{\cD_\eps}})-\psi(x)-\int_t^{s\wedge\tau_{\cD_\eps}}(\cL^*\psi)(r,Z_r)\ud r\Big]=0,
\end{equation}
with
\begin{equation}\label{eq:L*}
(\cL^* \psi)(t,z)\coloneqq\tfrac{1}{2}\mathrm{tr}(\sigma\sigma^\top(t,z) D^2 \psi(z))+\langle \mu(t,z)+\sigma(t,z) \alpha^*(t,z),\nabla \psi (z)\rangle.
\end{equation}
Thus, the law of the process $(Z_{r\wedge\tau_{\cD_\eps}})_{r\in[t,s]}$ under the measure $\M$ solves the martingale problem associated to the operator $\cL^*$. Then, on the filtered probability space $(\Omega',\cG,(\cG^Z_{r\wedge\tau_{\cD_\eps}})_{r\in[t,T]},\M)$ there is a Brownian motion $B$ such that the pair $(Z_{r\wedge\tau_{\cD_\eps}},B_{r\wedge\tau_{\cD_\eps}})_{r\in[t,T)}$ is a weak solution of \eqref{eq:X*loc3} (cf.\ \cite[Prop.\ 5.4.6]{KS}). However, the SDE in \eqref{eq:X*loc3} admits a unique strong solution (cf.\ Proposition \ref{prop:weakX})
and therefore for any $s\in[0,T)$ the law of the process $(Z_{r\wedge\tau_{\cD_\eps}})_{r\in[t,s]}$ is the same as the law of the process $(X^*_{r\wedge\rho^*_{\cD_\eps}})_{r\in[t,s]}$. We can let $\eps\to 0$ in \eqref{eq:X*b}, use dominated convergence and the bound in \eqref{eq:intAd2} along with the limit in \eqref{eq:intA} to extend the equivalence in law up to the stopping time $\tau_\cD$, because
\[
\lim_{\eps\to 0}\tau_{\cD_\eps}\wedge s=\tau_{\cD}\wedge s,\quad \P-a.s.
\]

The latter claim can be shown as follows. With probability one $\tau_{\cD_\eps}\le \tau_{\cD_{\eps'}}\le \tau_\cD$ for all $\eps'<\eps$. Then $\lim_{\eps\to 0}\tau_{\cD_\eps}\wedge s\le \tau_\cD\wedge s$, $\P$-a.s. In order to prove the reverse inequality we observe that on the event $\{\tau_\cD>s\}$ it must be $\inf_{0\le r\le s}\mathrm{dist}((r,Z_r),\cD)>0$ and therefore, by Assumption \ref{ass:u}-(i), for each $\omega\in\{\tau_\cD>s\}$ there is $\eps_\omega>0$ such that $u(r,Z_r(\omega))>\eps_\omega$, for all $r\in[t,s]$ (i.e., $(r,Z_r(\omega))\in\cC_{\eps_\omega}$ for all $r\in[t,s]$); that yields, $\lim_{\eps\to 0}\tau_{\cD_\eps}\wedge s = s$ on $\{\tau_\cD>s\}$. Instead, on the event $\{\tau_\cD\le s\}$, we argue differently. Let $\omega\in\{\tau_\cD\le s\}$. If $\tau_\cD(\omega)=t$, it is clear that $\lim_{\eps\to 0}\tau_{\cD_\eps}(\omega)\wedge s \ge \tau_\cD(\omega)\wedge s$. If $\tau_\cD(\omega)>t$, for any $t<\delta<\tau_\cD(\omega)$ we have $\inf_{0\le r\le \delta}\mathrm{dist}((r,Z_r(\omega)),\cD)>0$. Then, by Assumption \ref{ass:u}-(i) there is $\eps_\omega>0$ such that $u(r,Z_r(\omega))>\eps_\omega$, for all $r\in[t,\delta]$. That yields $\lim_{\eps\to 0}\tau_{\cD_\eps}(\omega)\wedge s \ge \delta\wedge s=\delta$ and letting $\delta\uparrow \tau_\cD(\omega)$ we deduce $\lim_{\eps\to 0}\tau_{\cD_\eps}(\omega)\wedge s \ge \tau_\cD(\omega)\wedge s=\tau_{\cD}(\omega)$. Since $\omega\in\{\tau_\cD\le s\}$ was arbitrary, combining the two events $\{\tau_\cD>s\}$ and $\{\tau_\cD\le s\}$ we conclude that $\lim_{\eps\to 0}\tau_{\cD_\eps}\wedge s = \tau_\cD(\omega)\wedge s$ as needed.

Since $s\in[0,T)$ was arbitrary, we conclude the proof of the lemma.
\end{proof}

\begin{proposition}\label{prop:limXT}
Fix $(t,x)\in\cC$ and let \eqref{eq:intA} hold. Then $\P((T,X^*_T)\in\cC_T)=1$ and therefore $\P((s,X^*_s)\in\cC,\,\forall s\in[t,T])=1$, by Proposition \ref{prop:weakX}.
\end{proposition}

\begin{proof}
The claim is trivial for $(t,x)\in\cC_T$, thus we may assume with no loss of generality that $(t,x)\in\cC_{[0,T)}$.

We prove  that $\P((T,X^*_T)\in\cD_T)=0$. First we observe that 
\[
\P\big((T,X^*_T)\in\cD_T\big)\le \P\big((T,X^*_T)\in\cD^\circ_T\big)+\P\big((T,X^*_T)\in\partial\cC_T\big),
\]
with the notation introduced just before Assumption \ref{ass:u}. By Assumption \ref{ass:u}-(i) we know that for every $\omega\in\{(T,X^*_T)\in\cD^\circ_T\}$ it must be 
$\lim_{s\uparrow T}u\big(s,X^*_s(\omega)\big)=0$,
and, therefore $\rho^*_{\eps}(\omega)< T$ for every $\eps>0$. Taking a sequence $(\eps_n)_{n\in\N}$ decreasing to zero, we have $\omega\in\cap_{n\in\N}\{\rho^*_{\eps_n}<T\}$. Then,
\begin{equation}\label{eq:X*Dcirc}
\P\big((T,X^*_T)\in\cD^\circ_T\big)\le \P\big(\cap_{n\in\N}\{\rho^*_{\eps_n}<T\}\big)=\lim_{n\to\infty}\P(\rho^*_{\eps_n}<T)=0,
\end{equation}
by \eqref{eq:intA}.

Next, we write the event $\{(T,X^*_T)\in\partial\cC_T\}$ as
\[
\big\{(T,X^*_T)\in\partial\cC_T\big\}=\bigcap_{n\in\N}\bigcup_{m\in\N}\bigcap_{k\ge m}\Big\{\mathrm{dist}\big((t_k,X^*_{t_k}),\partial\cC_T\big)\le\tfrac1n\Big\},
\]
where $(t_k)_{k\in\N}\subset[0,T)$ is an arbitrary sequence that increases (strictly) to $T$. By monotone convergence and Fatou's lemma for probability measures we have
\[
\P\big((T,X^*_T)\in\partial\cC_T\big)\le \lim_{n\to\infty}\liminf_{m\to\infty}\P\Big(\mathrm{dist}\big((t_m,X^*_{t_m}),\partial\cC_T\big)\le\tfrac1n\Big).
\]
For fixed $m\in\N$ we have by \eqref{eq:intA}
\[
\P\Big(\mathrm{dist}\big((t_m,X^*_{t_m}),\partial\cC_T\big)\le\tfrac1n\Big)=\lim_{\eps\to 0}\P\Big(\mathrm{dist}\big((t_m,X^*_{t_m}),\partial\cC_T\big)\le\tfrac1n,t_m<\rho^*_{\eps}\Big).
\]
Then, recalling the notation $\tau_{\cD_\eps}$ from the proof of Lemma \ref{lem:laws}, the same lemma implies 
\begin{equation}
\begin{aligned}
&\P\Big(\mathrm{dist}\big((t_m,X^*_{t_m}),\partial\cC_T\big)\le\tfrac1n,\, t_m<\rho^*_{\eps}\Big)\\
&=\M\Big(\mathrm{dist}\big((t_m,Z_{t_m}),\partial\cC_T\big)\le\tfrac1n,\, t_m<\tau_{\cD_\eps}\Big)\\
&=\frac{1}{u(t,x)}\E^\Q\Big[\Theta^\cD_{t_m\wedge\tau_{\cD_\eps}}1_{\big\{\mathrm{dist}\big((t_m,Z_{t_m}),\partial\cC_T\big)\le\tfrac1n,\, t_m<\tau_{\cD_\eps}\big\}}\Big]\\
&\le \frac{1}{u(t,x)}\Q\Big(\mathrm{dist}\big((t_m,Z_{t_m}),\partial\cC_T\big)\le\tfrac1n,\, t_m<\tau_{\cD_\eps}\Big)\\
&\le \frac{1}{u(t,x)}\Q\Big(\mathrm{dist}\big((t_m,Z_{t_m}),\partial\cC_T\big)\le\tfrac1n\Big),
\end{aligned}
\end{equation}
where in the penultimate inequality we use that $0\le u\le 1$ and $u(t,x)>0$.

Combining the bounds above we obtain
\begin{equation}
\begin{aligned}
\P\big((T,X^*_T)\in\partial\cC_T\big)
&\le\lim_{n\to\infty}\liminf_{m\to\infty}\frac{1}{u(t,x)}\Q\Big(\mathrm{dist}\big((t_m,Z_{t_m}),\partial\cC_T\big)\le\tfrac1n\Big)\\
&\le\lim_{n\to\infty}\limsup_{m\to\infty}\frac{1}{u(t,x)}\Q\Big(\mathrm{dist}\big((t_m,Z_{t_m}),\partial\cC_T\big)\le\tfrac1n\Big)\\
&\le\lim_{n\to\infty}\frac{1}{u(t,x)}\Q\Big(\limsup_{m\to\infty}\Big\{\mathrm{dist}\big((t_m,Z_{t_m}),\partial\cC_T\big)\le\tfrac1n\Big\}\Big).
\end{aligned}
\end{equation}
By continuity of paths of $(s,Z_s)$ we deduce
\[
\limsup_{m\to\infty}\Big\{\mathrm{dist}\big((t_m,Z_{t_m}),\partial\cC_T\big)\le\tfrac1n\Big\}\subseteq\Big\{\mathrm{dist}\big((T,Z_{T}),\partial\cC_T\big)\le\tfrac2n\Big\}.
\]
Hence, 
\[
\P\big((T,X^*_T)\in\partial\cC_T\big)\le \frac{1}{u(t,x)}\lim_{n\to\infty}\Q\Big(\mathrm{dist}\big((T,Z_{T}),\partial\cC_T\big)\le\tfrac2n\Big)=0,
\]
where the final equality holds by Assumption \ref{ass:u}-(ii).

Combining the latter bound with \eqref{eq:X*Dcirc} we have shown that $\P((T,X^*_T)\in\cD_T)=0$, thus concluding the proof of the proposition.
\end{proof}

Now we are ready to prove Theorem \ref{thm:main}. In its proof we are going to verify condition \eqref{eq:intA} so that we will be able to use Propositions \ref{prop:weakX} and \ref{prop:limXT}.

\begin{proof}[{\bf Proof of Theorem \ref{thm:main}}]
Recall $V=-2\ln u$ from Corollary \ref{cor:V}. 
Fix arbitrary $\eps>0$ and take $(t,x)\in\cC_{\eps}$. Recall the notation $\rho^*_{\eps,\delta}$ introduced after Proposition \ref{prop:strong}. The same proposition guarantees that the process $X^*_{s\wedge\rho^*_{\eps,\delta}}$ is the unique strong solution of \eqref{eq:X*loc} for $s\in[t,T)$ with $X^*_t=x$, for any $\delta,\eps>0$. 

For any $s\in[t,T)$ we simplify notation by setting $\tau^{\,s}_{\eps,\delta}\coloneqq s\wedge\rho^*_{\eps,\delta}$. Applying Dynkin's formula and using the definition of the control $\alpha^*(r,X^*_r)$ from \eqref{eq:a*}, we obtain 
\begin{equation}\label{eq:V1}
\begin{aligned}
V(t,x)&=\E\Big[V(\tau^{\,s}_{\eps,\delta},X^{*}_{\tau^{\,s}_{\eps,\delta}})\!-\!\int_{t}^{\tau^{\,s}_{\eps,\delta}}\!\!\big(\partial_t V\!+\!\cL V \!+\!\langle\nabla V, \sigma \alpha^*\rangle\big)(r,X^{*}_r) \ud r\Big]\\
&=\E\Big[V(\tau^{\,s}_{\eps,\delta},X^*_{\tau^{\,s}_{\eps,\delta}})\!-\!\int_{t}^{\tau^{\,s}_{\eps,\delta}}\!\!\Big(\!\big(\partial_t V\!+\!\cH(\cdot,\nabla V, D^2 V)\big)(r,X^*_r)\!-\!|\alpha^*(r,X^*_r)|^2\Big)\ud r\Big]\\
&= \E\Big[ 
V(\tau^{\,s}_{\eps,\delta},X^*_{\tau^{\,s}_{\eps,\delta}})\!+\!\int_{t}^{\tau^{\,s}_{\eps,\delta}}\!\Big(f(r,X^{*}_r)\!+\!|\alpha^*(r,X^*_r)|^2\Big)\ud r\Big],
\end{aligned}
\end{equation}
where we used \eqref{eq:PDEV} and \eqref{eq:Ham} for the second and third equality. 

For $s\in[t,T)$ we have $\eps \le u(\tau^{\, s}_{\eps,\delta},X^*_{\tau^{\, s}_{\eps,\delta}})\le 1$ by Assumption \ref{ass:u}--(i), thus $0\le V(\tau^{\, s}_{\eps,\delta},X^*_{\tau^{\, s}_{\eps,\delta}})\le -2\ln\eps$. 
Letting $\delta\to 0$ in \eqref{eq:V1} we have $\tau^{\,s}_{\eps,\delta}\uparrow \rho^*_{\eps}\wedge s$ (cf.\ \eqref{eq:rhoeps}). We can use dominated convergence and monotone convergence 
to pass to the limit in \eqref{eq:V1}, along with continuity of the mapping $u\mapsto V(s\wedge u\wedge\rho^*_{\eps},X^*_{s\wedge u\wedge\rho^*_{\eps}})$, which is guaranteed by Assumption \ref{ass:u}-(i). 
Then we get
\begin{equation}\label{eq:Veps}
\begin{aligned}
V(t,x)=\E\Big[ 
V(s\wedge\rho^*_{\eps},X^*_{s\wedge\rho^*_{\eps}})\!+\!\int_{t}^{s\wedge\rho^*_{\eps}}\!\!\Big(f(r,X^{*}_r)\!+\!|\alpha^*(r,X^*_r)|^2\Big)\ud r\Big].
\end{aligned}
\end{equation}
Moreoever,
\[
1_{\{\rho^*_{\eps}\le s\}}V(s\wedge\rho^*_{\eps},X^*_{s\wedge\rho^*_{\eps}})=1_{\{\rho^*_{\eps}\le s\}}V(\rho^*_{\eps},X^*_{\rho^*_{\eps}})=-2\ln\eps 1_{\{\rho^*_{\eps}\le s\}}.
\]
Plugging the above into \eqref{eq:Veps} and recalling $V,f\ge 0$ yields $V(t,x)\ge -2\ln \eps\, \P(\rho^*_{\eps}\le s)$. 
Letting $s\uparrow T$ yields $-2\ln \eps\, \P(\rho^*_{\eps}< T)\le V(t,x)$ and then, taking limits as $\eps\to 0$, we obtain 
\begin{equation}
\lim_{\eps\to 0}\P(\rho^*_{\eps}< T)=0.
\end{equation}
Thus, \eqref{eq:intA} holds and therefore Propositions \ref{prop:weakX} and \ref{prop:limXT} apply.

Letting $s\uparrow T$ in \eqref{eq:Veps} and using Fatou's lemma we obtain 
\begin{equation}
\begin{aligned}
V(t,x)&\ge \E\Big[ 
V(T\wedge\rho^*_{\eps},X^*_{T\wedge\rho^*_{\eps}})\!+\!\int_{t}^{T\wedge\rho^*_{\eps}}\!\!\Big(f(r,X^{*}_r)\!+\!|\alpha^*(r,X^*_r)|^2\Big)\ud r\Big]\\
&\ge \E\Big[ 
1_{\{\rho^*_{\eps}>T\}}g(X^*_{T})\!+\!\int_{t}^{T\wedge\rho^*_{\eps}}\!\!\Big(f(r,X^{*}_r)\!+\!|\alpha^*(r,X^*_r)|^2\Big)\ud r\Big],
\end{aligned}
\end{equation}
where the first inequality holds because $(s,X^*_s)\in\cC$ for all $s\in[t,T]$ and $V\in C(\cC)$ and
the second one holds because $V\ge 0$. By Proposition \ref{prop:limXT}, we have $\P(\rho^*=T)=1$ and $\rho^*_{\eps}\uparrow \rho^*$ when $\eps\downarrow 0$ (possibly along a subsequence). We use Fatou's lemma once again to get
\begin{equation}\label{eq:Veps2}
\begin{aligned}
V(t,x)\ge \E\Big[ 
g(X^*_{T})\!+\!\int_{t}^{T}\!\!\Big(f(r,X^{*}_r)\!+\!|\alpha^*(r,X^*_r)|^2\Big)\ud r\Big]\ge v(t,x).
\end{aligned}
\end{equation}

Next we prove the reverse inequality.
We define $V^\lambda(t,x)\coloneqq-2\ln (u(t,x)+\lambda)$ for $\lambda\in (0,1)$. Notice that $V^\lambda\in C^{1,2}(\cC_{[0,T)})\cap C(\cC)$ and $V^\lambda\uparrow V$ as $\lambda\downarrow 0$. Moreover, the convergence is uniform on compact subsets of $\cC$ by Dini's theorem because both $V^\lambda$ and $V$ are continuous on $\cC$. 
From direct calculations it is not difficult to show that 
\begin{equation}\label{eq:HJBeps}
\left\{
\begin{array}{ll}
\partial_t V^\lambda(t,x)+\cH\big(t,x,\nabla V^\lambda(t,x),D^2 V^\lambda(t,x)\big)+\big(\frac{u}{u+\lambda}f\big)(t,x)=0,& (t,x)\in\cC_{[0,T)},\\[+3pt]
V^\lambda(T,x)=-2\ln(\e^{-\frac12 g(x)}+\lambda), & (T,x)\in\cC_T,\\[+3pt]
-2\ln (1+\lambda)\le V^\lambda(t,x)\le -2\ln \lambda, &(t,x)\in\cO,
\end{array}
\right.
\end{equation}
where we recall
\[
-\tfrac{1}{4}\big|\sigma^\top\nabla V^\lambda \big|^2(t,x)=\inf_{\alpha\in\R^{d'}}\big(|\alpha|^2+\langle\nabla V^\lambda,\sigma\alpha\rangle\big)(t,x).
\]
Notice that by definition $V^\lambda =-2\ln \lambda$ on $\cD$. 

Let $a\in\cA^\cD_{t,x}$ be arbitrary (cf.\ Definition \ref{def:admissible}). Recall the dynamics of $X^a$ in \eqref{eq:Xa} and notice that by admissibility it must be $(s,X^{a}_s)\in \cC$ for all $s\in[t,T]$, $\P$-a.s. 
Let 
\[
\rho^a_{\eps,\delta}\coloneqq\inf\{s\in[t,T]:(s,X^a_s)\notin\cC_{\eps,\delta}\}\wedge T\quad\text{and}\quad\rho^a_{\eps}\coloneqq\inf\{s\in[t,T]:(s,X^a_s)\notin\cC_{\eps}\}\wedge T,
\] 
with $\inf\varnothing=+\infty$.  
Since $X^a_t=x$, an application of Dynkin's formula and \eqref{eq:HJBeps} give, for $s\in[t,T)$
\begin{equation}
\begin{aligned}
V^\lambda(t,x)
=&\E\Big[V^\lambda(s\wedge\rho^a_{\eps,\delta},X^{a}_{s\wedge\rho^a_{\eps,\delta}})\Big]\\
&-\E\Big[\int_{t}^{s\wedge\rho^a_{\eps,\delta}}\Big(\big(\partial_t V^\lambda+\cL V^\lambda\big)(r,X^{a}_r)+\langle\nabla V^\lambda(r,X^{a}_r),\sigma(r,X^{a}_r) a_r\rangle\Big)\ud r\Big]\\
\le& \E\Big[V^\lambda(s\wedge\rho^a_{\eps,\delta},X^{a}_{s\wedge\rho^a_{\eps,\delta}})\!-\!\int_{t}^{s\wedge\rho^a_{\eps,\delta}}\!\!\Big(\!\big(\partial_t V^\lambda\!+\!\cH\big(\cdot,\nabla V^\lambda,D^2 V^\lambda\big) \big)(r,X^{a}_r)\!-\!|a_r|^2\Big)\ud r\Big]\\
=& \E\Big[V^\lambda(s\wedge\rho^a_{\eps,\delta},X^{a}_{s\wedge\rho^a_{\eps,\delta}})\!+\!\int_{t}^{s\wedge\rho^a_{\eps,\delta}}\!\!\Big(\big(\frac{u}{u+\lambda}f\big)(r,X^{a}_r)\!+\!|a_r|^2\Big)\ud r\Big].
\end{aligned}
\end{equation}
By continuity of $u$ (hence of $V^\lambda$) in $\cC$ and admissibility of $a\in\cA_{t,x}^\cD$, letting first $s\to T$ and then $\delta,\eps\to 0$ we get $\rho^a_{\eps,\delta}\wedge s \uparrow T$ with $(T,X_T^{a})\in\cC_T$ and therefore
\[
\lim_{\delta,\eps\to 0}\lim_{s\to T}V^\lambda(s\wedge\rho^a_{\eps,\delta},X^{a}_{s\wedge\rho^a_{\eps,\delta}})=-2\ln\big( \e^{-\frac12 g(X^{a}_T)}+\lambda\big),\quad\P-a.s.
\]
Thanks to boundedness of $V^\lambda$ and positivity of the integral, letting first $s\to T$ and then $\delta,\eps\to 0$, we apply dominated convergence and monotone convergence to arrive at 
\begin{equation}
\begin{aligned}
V^\lambda(t,x)
\le \E\Big[ 
-2\ln\big( \e^{-\frac12 g(X^{a}_T)}+\lambda\big)\!+\!\int_{t}^{T}\!\!\Big(\big(\frac{u}{u+\lambda}f\big)(r,X^{a}_r)\!+\!|a_r|^2\Big)\ud r\Big].
\end{aligned}
\end{equation}
Finally, letting $\lambda\downarrow 0$ we invoke monotone convergence to obtain
\[
V(t,x)\le \E\Big[ g(X^{t,x;a}_T)\!+\!\int_{t}^{T}\!\!\Big(f(r,X^{a}_r)\!+\!|a_r|^2\Big)\ud r\Big]
\]
and by arbitrariness of $a\in\cA^{\cD}_{t,x}$ we obtain $V\le v$ as needed. Combining that with \eqref{eq:Veps2} gives $V=v$ and optimality of the dynamics $(X^*_s)_{s\in[t,T]}$.
\end{proof}

\section{Sufficient conditions for the continuity of $u$}\label{sec:suffcond}

Assumption \ref{ass:u}-(i) may be difficult to check in practice, when the form of the function $u$ is not known explicitly. In those cases, a sufficient condition would be the continuity of the function $u$ in $\mathcal O$. Unfortunately that requirement may be too restrictive
because $u(t,z)=0$ for $(T,z)\in\cD_T$ but $u(T,z)=\exp(-\frac12 g(z))$ for $(T,z)\in\cC_T$. Unless $g(z)\to\infty$ for $z$ approaching any $z_0$ such that $(T,z_0)\in\cD_T\cap\overline{\cC_T}=\partial \cC_T$, there is always going to be a discontinuity at points of $\partial\cC_T$ (compare with Examples \ref{ex:1} and \ref{ex:2}). 

In this section we provide some results that address this issue. Recall the notations $\cD^\circ$, $\cD_{[0,T)}^\circ$ and $\cD^\circ_T$ introduced in \eqref{eq:setnot}. 
Further, we denote 
\[
\partial\cD=\overline\cC\cap\cD\quad\text{and}\quad\partial\cD_{[0,T)}=(\overline\cC\cap\cD)_{[0,T)}.
\]
\begin{proposition}\label{prop:suffcond}
Assume $f\in C(\overline\cC)$ and $g\in C(\overline{\cC_T})$. Assume that 
for any $\delta>0$ and any sequence $(t_n,z_n)_{n\in\N}\subset\cO$ converging to $(t,z)\in \cO$
\begin{equation}\label{eq:Zcont}
\lim_{n\to\infty}\E^\Q\Big[\sup_{s\in[0,\delta]}\big|Z^{t_n,z_n}_{(t_n+s)\wedge T}-Z^{t,z}_{(t+s)\wedge T}\big|\Big]=0.
\end{equation}
Assume further that Assumption \ref{ass:u}-(ii) holds.
\begin{itemize}
\item[(a)] If $(t,z)\in\cC$ is such that 
\begin{equation}\label{eq:regD}
\lim_{n\to\infty}\Q\big(\big|\tau^{t_n,z_n}_\cD\wedge (2T)-\tau^{t,z}_\cD\wedge (2T)\big|>\eps\big)=0,\quad \forall \eps>0,
\end{equation}
for any sequence $(t_n,z_n)_{n\in\N}\subset\cC$ converging to $(t,z)$, 
then $\lim_{n\to\infty}u(t_n,z_n)=u(t,z)$. 

\item[(b)] If $(t,z)\in\partial\cD_{[0,T)}\cup\cD^\circ_T$ is such that \eqref{eq:regD} holds for any any sequence $(t_n,z_n)_{n\in\N}\subset\cC$ converging to $(t,z)$, then $\lim_{n\to\infty}u(t_n,z_n)=u(t,z)=0$.
\end{itemize}

If \eqref{eq:regD} holds for any $(t,z)\in \cC\cup\partial \cD_{[0,T)}\cup \cD^\circ_T$ and any sequence $(t_n,z_n)_{n\in\N}\subset\cC$ converging to $(t,z)$, then Assumption \ref{ass:u}-(i) holds. 
\end{proposition}

For clarity we notice that in Example \ref{ex:1} it is indeed possible to construct sequences $(t_n,z_n)_{n\in\N}\subset\cC$ converging to $(t,z)\in\cD^\circ_T$, whereas this is not possible in, e.g., Example \ref{ex:2}. We also notice that, by definition, either $\tau_\cD(\omega)\le T$ or $\tau_\cD(\omega)=\infty$; then, the limit in probability in \eqref{eq:regD} should be read as saying that conditionally on $\tau^{t,z}_\cD=\infty$ (or on $\tau^{t,z}_\cD<\infty$), the limit in probability of $\tau^{t_n,z_n}_\cD$ is infinite (or $\tau^{t,z}_\cD$). The choice of $2T$ in \eqref{eq:regD} is arbitrary; any $T'>T$ would equally serve the purpose.

\begin{proof}[{\bf Proof of Proposition \ref{prop:suffcond}}]
The argument of proof is the same under (a) and (b). Thus we fix $(t,z)\in\cC\cup \partial\cD_{[0,T)}\cup\cD^\circ_T$ and take a sequence $(t_n,z_n)_{n\in\N}\subseteq\cC$ converging to $(t,z)$. 
Arguing by contradiction let us assume that 
\begin{equation}\label{eq:contrad}
\limsup_{n\to\infty}\big|u(t_n,z_n)-u(t,z)\big| > 0.
\end{equation}
With no loss of generality, up to selecting a subsequence, we can assume that the ``$\limsup$'' is actually a ``$\lim$'' and so we do that without further mentioning.

Thanks to \eqref{eq:Zcont} and \eqref{eq:regD}, there is $\Omega_0\in\cF$ with $\Q(\Omega_0)=1$ and a subsequence $(t_{n_k},z_{n_k})_{k\in\N}$ for which
\begin{equation}\label{eq:tauconv}
\lim_{k\to\infty}\sup_{s\in[0,\delta]}\big|Z^{t_{n_k},z_{n_k}}_{t_{n_k}+s}(\omega)-Z^{t,z}_{t+s}(\omega)\big|=0\quad\mbox{and}\quad \lim_{k\to\infty}\tau^{t_{n_k},z_{n_k}}_\cD(\omega)=\tau^{t,z}_\cD(\omega),
\end{equation}
for\ all\ $\omega\in\Omega_0$.
Set $\tau^k_\cD=\tau^{t_{n_k},z_{n_k}}_\cD$, $Z^{k}=Z^{t_{n_k},z_{n_k}}$, $\tau_\cD=\tau_\cD^{t,z}$ and $Z=Z^{t,z}$ for simplicity. Then, for each $k\in\N$
\begin{equation}\label{eq:uku}
\begin{aligned}
&\big|u(t_{n_k},z_{n_k})-u(t,z)\big|\\
&\le\Big|\E^\Q\Big[\Big(\e^{-\frac{1}{2}\int_{t_{n_k}}^T f(s,Z^k_s)\ud s-\frac{1}{2}g(Z^k_T)}-\e^{-\frac{1}{2}\int_{t}^T f(s,Z_s)\ud s-\frac{1}{2}g(Z_T)}\Big)1_{\{T<\tau^k_\cD\wedge\tau_\cD\}}\Big]\Big|\\
&\quad+\E^\Q\Big[\e^{-\frac{1}{2}\int_{t_{n_k}}^T f(s,Z^k_s)\ud s-\frac{1}{2} g(Z^k_T)}1_{\{\tau_\cD\le T<\tau^k_\cD\}}\Big]\\
&\quad+\E^\Q\Big[\e^{-\frac{1}{2}\int_{t}^T f(s,Z_s)\ud s-\frac{1}{2} g(Z_T)}1_{\{\tau^k_\cD\le T<\tau_\cD\}}\Big].
\end{aligned}
\end{equation}
By dominated convergence we have
\begin{equation}
\begin{aligned}
&\lim_{k\to\infty}\Big|\E^\Q\Big[\Big(\e^{-\frac{1}{2}\int_{t_{n_k}}^T f(s,Z^k_s)\ud s-\frac{1}{2}g(Z^k_T)}-\e^{-\frac{1}{2}\int_{t}^T f(s,Z_s)\ud s-\frac{1}{2}g(Z_T)}\Big)1_{\{T<\tau^k_\cD\wedge\tau_\cD\}}\Big]\Big|\\
&\le \lim_{k\to\infty}\E^\Q\Big[\Big|\e^{-\frac{1}{2}\int_{t_{n_k}}^T  f(s,Z^k_s)\ud s-\frac{1}{2} g(Z^k_T)}-\e^{-\frac{1}{2}\int_{t}^T  f(s,Z_s)\ud s-\frac{1}{2} g(Z_T)}\Big|1_{\{T<\tau^k_\cD\wedge\tau_\cD\}}\Big]\\
&= \E^\Q\Big[\lim_{k\to\infty}\Big|\e^{-\frac{1}{2}\int_{t_{n_k}}^T  f(s,Z^k_s)\ud s-\frac{1}{2} g(Z^k_T)}-\e^{-\frac{1}{2}\int_{t}^T  f(s,Z_s)\ud s-\frac{1}{2} g(Z_T)}\Big|1_{\{T<\tau^k_\cD\wedge\tau_\cD\}}\Big]=0.
\end{aligned}
\end{equation}

For the second and third term on the right-hand side of \eqref{eq:uku} we argue in a different way. 
We recall that $\{T<\tau_\cD\}=\{\tau_\cD=\infty\}$ and $\{T<\tau^k_\cD\}=\{\tau^k_\cD=\infty\}$, by definition of $\tau_\cD$ and $\tau_\cD^k$. Using this fact we have, for $0<\eps<T$
\begin{equation}\label{eq:2nd}
\begin{aligned}
&\E^\Q\Big[\e^{-\frac{1}{2}\int_{t_{n_k}}^T f(s,Z^k_s)\ud s-\frac{1}{2} g(Z^k_T)}1_{\{\tau_\cD\le T<\tau^k_\cD\}}\Big]\\
&\le\Q\big(\tau_\cD\le T<\tau^k_\cD\big)\\
&= \Q\big(\tau_\cD\le T, \tau^k_\cD=\infty\big)\le \Q\big(|\tau_\cD\wedge(2T)-\tau^k_\cD\wedge(2T)|>\eps\big),
\end{aligned}
\end{equation}
and analogously
\begin{equation}\label{eq:3rd}
\begin{aligned}
&\E^\Q\Big[\e^{-\frac{1}{2}\int_{t}^T f(s,Z_s)\ud s-\frac{1}{2} g(Z_T)}1_{\{\tau^k_\cD\le T<\tau_\cD\}}\Big]\\
&\le\Q\big(\tau^k_\cD\le T<\tau_\cD\big)\\
&= \Q\big(\tau^k_\cD\le T, \tau_\cD=\infty\big)\le \Q\big(|\tau_\cD\wedge(2T)-\tau^k_\cD\wedge(2T)|>\eps\big).
\end{aligned}
\end{equation}

Letting $k\to\infty$ in \eqref{eq:2nd} and \eqref{eq:3rd} we reach a contradiction with \eqref{eq:contrad}. 
Then 
\[
\lim_{n\to\infty}u(t_n,z_n)=u(t,z),
\]
as needed for (a) and (b). If \eqref{eq:regD} holds for any $(t,z)\in\cC\cup\partial\cD_{[0,T)}\cup\cD^\circ_T$, the results from (a) and (b) together with $u=0$ on $\cD$ imply continuity of $u$ on $\cO\setminus\partial\cC_T$. Then Assumption \ref{ass:u}-(i) holds.
\end{proof}

\begin{remark}\label{rem:suff}
Some comments about the assumptions in the proposition above are in order:
\begin{itemize}
\item[(a)] Condition \eqref{eq:Zcont} holds if, for example, the coefficients $\mu$ and $\sigma$ are Lipschitz continuous in the spatial variable, uniformly for $t\in[0,T]$, thanks to standard estimates on SDEs \cite[Thm.\ 2.5.9]{K}. 

\item[(b)] The condition in \eqref{eq:regD} is closely related to a notion of regularity of the set $\cD$ in the sense of diffusions and it may be checked on a case-by-case basis. Numerous sufficient conditions for \eqref{eq:regD} exist and a summary of known results with appropriate references can be found in \cite[Sec.\ 2]{DeAP20}. 
\end{itemize}
\end{remark}

\subsection{Examples of regular sets $\cD$}
For completeness we close the section with a study of some sufficient conditions for \eqref{eq:regD}. We divide the analysis into two main cases: in one case the uncontrolled process $(t,Z)$ cannot hit the forbidden set $\cD$ exactly at time $T$ (cf.\ Assumption \ref{ass:interior} for a formal statement) and if it hits $\cD_{[0,T)}$ it also immediately enters $\cD^\circ_{[0,T)}$, as in Example~\ref{ex:2}; in the other case, the process $(t,Z)$ can hit the interior set $\cD^\circ_T$ but the overall geometry of $\cD$ is rather special as in, e.g., Examples \ref{ex:1} and \ref{ex:3}.

The first {\em hitting} time of a set $A\subset \cO$ for the process $(s,Z_{s})_{s\in[t,T]}$ is denoted 
\[
\sigma_A=\inf\{s\in(t,T]:(s,Z_s)\in A\},
\]
with our usual convention $\inf\varnothing=\infty$.
In line with standard terminology from the theory of Markov processes (cf.\ \cite{BG}) we say that a set $A\subset \cO$ is regular in the sense of diffusions if $\Q_{t,z}(\sigma_A=t)=1$ for all $(t,z)\in\overline A$
with $t<T$. Notice that we require that the initial point $(t,z)$ be such that $t<T$ because the section $A_T$ cannot be regular as there is no dynamics after time $T$ in our problem.

First we focus on the case when the uncontrolled process $(t,Z)$ cannot hit the portion $\cD^\circ_T$ of the forbidden set $\cD$.
\begin{assumption}\label{ass:interior}
For $(T,z)\in\cD^\circ_T$ there is $\eps>0$ and a ball $B_\eps(z)$ in $\R^d$ of radius $\eps$ and centered in $z$ such that $[T-\eps,T]\times B_\eps(z)\subset\cD$.
\end{assumption}
Assumption \ref{ass:interior} is not met in the Examples \ref{ex:1} and \ref{ex:3}. We will look at that type of setups in a separate statement. 
\begin{proposition}\label{prop:suffcond2}
Assume that \eqref{eq:Zcont} and Assumption \ref{ass:u}-(ii) hold. If Assumption \ref{ass:interior} holds and the set $\cD^\circ_{[0,T)}$ is regular in the sense of diffusions for the process $(s,Z_s)$, then \eqref{eq:regD} holds for any sequence $(t_n,z_n)_{n\in\N}\subset\cC_{[0,T)}$ converging to $(t,z)\in\overline{\cC}_{[0,T)}\cup\cC_{T}\cup\cD^\circ_T$. 
\end{proposition}

\begin{proof}
First of all we notice that by Assumption \ref{ass:interior} and continuity of paths of $Z$, for $(t,z)\in\cC$ we have
\[
\{\tau^{t,z}_\cD=T\}=\{(s,Z^{t,z}_s)\in\cC,s\in[t,T)\}\cap\{(T,Z^{t,z}_T)\in\cD_T\}\subset\{(T,Z^{t,z}_T)\in\partial\cC_T\}.
\]
Then, Assumption \ref{ass:u}-(ii) implies 
\begin{equation}\label{eq:Qtau}
\Q_{t,z}(\tau_\cD=T)=0,\quad \text{for $(t,z)\in\cC$}.
\end{equation}

If $(T,z)\in\cD^\circ_T$, Assumption \ref{ass:interior} implies that for any sequence $(t_n,z_n)\to (T,z)$ there is a large enough index $N\in\N$ so that $(t_n,z_n)\in\cD$ for all $n\ge N$. That implies $\tau^{t_n,z_n}_{\cD}(\omega)=t_n$ for all $n\ge N$ and all $\omega\in\Omega$. Therefore, $\tau^{t_n,z_n}_{\cD}\to T$ a.s.\ as $n\to\infty$, proving that \eqref{eq:regD} holds for $(T,z)\in\cD^\circ_T$. 

Next we consider $(t,z)\in\overline\cC_{[0,T)}\cup \cC_T$. 
Letting $\tau^{t,z}_{\cD^\circ}=\inf\{s\in[t,T]:(s,Z^{t,z}_s)\in\cD^\circ\}$ with the convention $\inf\varnothing=\infty$, 
it is clear that $\tau_{\cD}\le\tau_{\cD^\circ}$. 
Notice that 
\begin{equation}\label{eq:taud0}
\Q_{t,z}(\tau_{\cD}=\tau_{\cD^\circ}=0)=1,\quad\text{for $(t,z)\in\cD_{[0,T)}$},
\end{equation} 
by regularity in the sense of diffusions of $\cD^\circ_{[0,T)}$. It is also clear that $\Q_{T,z}(\tau_{\cD}=\tau_{\cD^\circ}=\infty)=1$ for $(T,z)\in\cC_T$.
Let us now consider $(t,z)\in\cC_{[0,T)}$. Arguing by contradiction let us assume that $\Q_{t,z}(\tau_{\cD^\circ}>\tau_{\cD})>0$. Notice that $\{\tau_{\cD^\circ}>\tau_{\cD}\}\subset\{\tau_\cD<T\}$, $\Q_{t,z}$-a.s., thanks to \eqref{eq:Qtau}. We can find $\delta>0$ such that $\Q_{t,z}(\tau_{\cD^\circ}\ge \tau_{\cD}+\delta)>0$. 
Moreover, by continuity of paths, 
we have  $(\tau_\cD,Z_{\tau_{\cD}})\in\partial\cD_{[0,T)}$, $\Q_{t,z}$-a.s.\ on the event $\{\tau_\cD<T\}$.
Combining these two facts with strong Markov property we obtain
\begin{equation}
\begin{aligned}
0 &<\Q_{t,z}(\tau_{\cD^\circ}\ge \tau_{\cD}+\delta)=\E^\Q_{t,z}\Big[\Q_{t,z}\big(\tau_{\cD^\circ}\ge \tau_{\cD}+\delta\big|\cF_{\tau_\cD}\big)1_{\{\tau_\cD<T\}}\Big]\\
&=\E^\Q_{t,z}\Big[1_{\{\tau_\cD<T\}}\Q_{\tau_\cD,Z_{\tau_\cD}}\big(\tau_{\cD^\circ}\ge \delta\big)\Big]=0,
\end{aligned}
\end{equation} 
because by regularity of $\cD^\circ_{[0,T)}$ it holds $\Q_{\tau_\cD,Z_{\tau_\cD}}\big(\tau_{\cD^\circ}=0\big)=1$, $\Q_{t,z}$-a.s.\ on $\{\tau^{t,z}_\cD<T\}$, by \eqref{eq:taud0}. Hence we reach a contradiction and we conclude that  
\begin{equation}\label{eq:taudintd}
\Q_{t,z}(\tau_{\cD}=\tau_{\cD^\circ})=1,
\end{equation}
for all $(t,z)\in\cC_{[0,T)}\cup\cD_{[0,T)}\cup\cC_T$.
By continuity of paths of $Z$ and regularity of $\cD^\circ_{[0,T)}$ it follows that for $(t,z)\in\cC$ the entry time and the {\em hitting time} to $\cD^\circ$ coincide, i.e.,
\begin{equation}\label{eq:reg2}
\Q_{t,z}\big(\tau_{\cD^\circ}=\sigma_{\cD^\circ}\big)=1,\quad(t,z)\in\cC,
\end{equation}
where 
\[
\tau_{\cD^\circ}=\inf\{s\in[t,T]:(s,Z_s)\in\cD^\circ\}\quad\text{and}\quad\sigma_{\cD^\circ}=\inf\{s\in(t,T]:(s,Z_s)\in\cD^\circ\},
\]
with $\inf\varnothing=\infty$. 

Thanks to \eqref{eq:Zcont}, for any sequence $(t_n,z_n)_{n\in\N}\subset \cC$ converging to  $(t,z)\in \overline\cC_{[0,T)}\cup\cC_T$ we can extract a subsequence $(t_{n_k},z_{n_k})_{k\in\N}$ for which
\begin{equation}\label{eq:Zas}
\lim_{k\to\infty}\sup_{s\in[0,\delta]}\big|Z^{t_{n_k},z_{n_k}}_{t_{n_k}+s}(\omega)-Z^{t,z}_{t+s}(\omega)\big|=0,\quad \mbox{for all } \omega\in\Omega_0,
\end{equation}
where $\Omega_0\in\cF$ is such that $\Q(\Omega_0)=1$ (notice that $\Omega_0$ may depend on the original sequence). For the ease of notation let us relabel $Z^k=Z^{n_k}$, $\tau^k_\cD=\tau^{t_{n_k},z_{n_k}}_{\cD}$, $\tau^k_{\cD^\circ}=\tau^{t_{n_k},z_{n_k}}_{\cD^\circ}$, $\sigma^k_\cD=\sigma^{t_{n_k},z_{n_k}}_{\cD}$ and $\sigma^k_{\cD^\circ}=\sigma^{t_{n_k},z_{n_k}}_{\cD^\circ}$. It then holds, 
\begin{equation}\label{eq:limsup}
\limsup_{k\to\infty}\tau^{k}_\cD=\limsup_{k\to\infty}\tau^{k}_{\cD^\circ}
=\limsup_{k\to\infty}\sigma^{k}_{\cD^\circ}\le \sigma^{t,z}_{\cD^\circ}=\tau^{t,z}_{\cD^\circ}=\tau^{t,z}_{\cD},\quad\Q-a.s.,
\end{equation}
where we use \eqref{eq:taudintd} for the first and last equalities, \eqref{eq:reg2} for the second equality and the inequality can be proven as follows. Fix $\omega\in\Omega_0$. If $\sigma^{t,z}_{\cD^\circ}(\omega)=\infty$ then the inequality is trivial. If $\sigma^{t,z}_{\cD^\circ}(\omega)<\infty$, then it must be $\sigma^{t,z}_{\cD^\circ}(\omega)<T$ by \eqref{eq:Qtau} and regularity of $\cD^\circ_{[0,T)}$ in the sense of diffusions. Since $\cD^\circ$ is open, for any sufficiently small $\eps>0$ we have $(s,Z^{t,z}_{s}(\omega))\in\cD^\circ$ for $s\in(\sigma^{t,z}_{\cD^\circ}(\omega),\sigma^{t,z}_{\cD^\circ}(\omega)+\eps)$. It follows from \eqref{eq:Zas} that also $(s,Z^{k}_{s}(\omega))\in\cD^\circ$ for $s\in(\sigma^{t,z}_{\cD^\circ}(\omega),\sigma^{t,z}_{\cD^\circ}(\omega)+\eps)$, and therefore $\sigma^k_{\cD^\circ}(\omega)\le \sigma^{t,z}_{\cD^\circ}(\omega)+\eps$, for $k$ sufficiently large. Then
\[
\limsup_{k\to\infty}\sigma^{k}_{\cD^\circ}(\omega)\le \sigma^{t,z}_{\cD^\circ}(\omega)+\eps,
\]
and the claim is proven by letting $\eps\to 0$.

It can also be shown (following arguments as in, e.g., \cite[Lemma 1.2]{M80}) that 
\begin{equation}\label{eq:liminf}
\liminf_{k\to\infty}\tau^{k}_\cD\ge \tau^{t,z}_{\cD},\quad \Q-a.s.
\end{equation}
We repeat the full argument here due to subtle differences with the setup in \cite{M80}. 
Since $\Q(\tau^{k}_\cD=T)=0$ for all $k\in\N$ (cf.\ \eqref{eq:Qtau}), then we also have $\Q(\tau^{k}_\cD=T\text{ for some $k\in\N$})=0$. Thus, with no loss of generality we assume
that 
\begin{equation}\label{eq:omega0}
\omega\in\Omega_0\implies \omega\in\bigcap_{k\in\N}\{\tau^{k}_\cD<T\}\cup\{\tau^{k}_\cD=\infty\}.
\end{equation}

Fix $\omega\in\Omega_0$. If $\tau^{t,z}_\cD(\omega)=0$, then \eqref{eq:liminf} holds trivially. If $\tau^{t,z}_\cD(\omega)>0$, let $\eta>0$ be arbitrary and such that $\tau^{t,z}_\cD(\omega)>\eta$. Since $\cD$ is closed and sample paths are continuous, then 
\[
c_\omega\coloneqq\inf_{t\le s \le \eta\wedge T}\mathrm{dist}\big((s,Z^{t,z}_s(\omega)),\cD\big)>0.
\]
Thanks to \eqref{eq:Zas}, there is $K_\omega\in\N$ such that  
\[
\inf_{t_{n_k}\le s \le \eta\wedge T}\mathrm{dist}\big((s,Z^{k}_s(\omega)),\cD\big)\ge \tfrac12 c_\omega>0,\quad\ \text{for\ all}\ k\ge K_\omega.
\]
Hence $\tau^{k}_\cD(\omega)\ge \eta\wedge T$ for all $k\ge K_\omega$. By \eqref{eq:omega0} it is clear that when $\eta\ge T$ it must be $\tau^{k}_\cD(\omega)=\infty$ for all $k\ge K_\omega$ and therefore $\liminf_{k\to\infty}\tau^{k}_\cD(\omega)\ge \tau^{t,z}_\cD(\omega)$ holds trivially. Instead, for $\tau_\cD(\omega)<T$ and $\eta<T$ we have
\[
\theta(\omega)\coloneqq\liminf_{k\to\infty}\tau^{k}_\cD(\omega)\ge \eta.
\]
Letting $\eta\uparrow\tau^{t,z}_\cD(\omega)$ we obtain $\theta(\omega)\ge \tau^{t,z}_\cD(\omega)$ for all $\omega\in\Omega_0$. Hence, \eqref{eq:liminf} holds.

Combining \eqref{eq:limsup} and \eqref{eq:liminf} we obtain that, from any sequence $(t_n,z_n)_{n\in\N}$ converging to $(t,z)\in\overline\cC_{[0,T)}\cup\cC_T$ we can extract a sub-sequence $(t_{n_k},z_{n_k})_{k\in\N}$ such that 
\begin{equation}\label{eq:limtau} 
\lim_{k\to\infty}\tau^{t_{n_k},z_{n_k}}_{\cD}=\tau^{t,z}_\cD,\quad\Q-a.s.
\end{equation}
We now show that the latter implies \eqref{eq:regD}.

Arguing by contradiction, let us assume the existence of a sequence $(t_n,z_n)_{n\in\N}\subset\cC$ converging to $(t,z)\in\overline\cC_{[0,T)}\cup\cC_T$ but such that 
\begin{equation}\label{eq:contr2}
\limsup_{n\to\infty}\Q\big(\big|\tau^{t_n,z_n}_\cD\wedge (2T)-\tau^{t,z}_\cD\wedge (2T)\big|>\eps\big)>0,
\end{equation}
for some $\eps>0$. Up to selecting a subsequence we can replace ``$\limsup$'' with ``$\lim$'' and we do that without further mentioning. Then, in particular, \eqref{eq:contr2} holds along any further subsequence. This is however impossible because we know that there is at least one subsequence along which \eqref{eq:limtau} holds. Thus \eqref{eq:regD} holds. 
\end{proof}

\begin{remark}\label{rem:suff2}
Regularity of $\cD^\circ_{[0,T)}$ for the process $(s,Z_s)$ is guaranteed if the boundary of the set is, for example, Lipschitz continuous and the diffusion coefficient $\sigma$ is strictly elliptic at points on the boundary (i.e., for $(t,z)\in\partial\cD_{[0,T)}$ there is a constant $c=c(t,z)$ such that $\langle\sigma\sigma^\top (t,z)\nu,\nu\rangle\ge c|\nu|^2$ for any $\nu\in\R^d$). The Lipschitz condition can be substantially relaxed in several situations of interest. For example, if $Z\in\R$ and $\sigma(t,z)>0$, then it is sufficient that $\cD_{[0,T)}=\{(t,z):z\in(b(t),c(t))\}$ for some functions $b$ and $c$ that are non-decreasing and non-increasing, respectively. If instead $Z\in\R^2$ and the diffusion coefficient $\sigma$ is strictly elliptic at points of the boundary of $\cD_{[0,T)}$, then it is sufficient that $\cD_{[0,T)}=\{(t,z_1,z_2):z_2\in(b(t,z_1),c(t,z_1))\}$ for functions $b$ and $c$ that have a fixed monotonicity in $z_1$ for every $t$ (e.g., non-decreasing in $z_1$ for every $t$) and are non-decreasing and non-increasing, respectively, as functions of $t$. 
\end{remark}

Now we turn our attention to setups with the flavour of Examples \ref{ex:1} and \ref{ex:3}, which however may not be amenable to explicit calculations for the function $u$. We make the following assumption on the structure of $\cD$ and on the law of the process $(Z_t)_{t\in[0,T]}$.
\begin{assumption}\label{ass:ball}
Let $A_0\subset\R^d$ be open and let $t_0\in(0,T]$. We assume $\cD=\{t_0\}\times \overline A_0$ and
\begin{equation}\label{eq:Zt0}
\Q_{t,z}(Z_{t_0}\in\partial A_0)=0,\quad \text{for all $(t,z)\in[0,t_0)\times\R^d$},
\end{equation}
with $\partial A_0=\overline A_0\setminus A_0$. We also denote $\cD^\circ_{t_0}=\{t_0\}\times A_0$.
\end{assumption}
By the assumption we have $\Q_{t,z}(\tau_\cD=\infty)=1$ for $(t,z)\in([t_0,T]\times \R^d)\setminus \cD$. Then, 
\[
u(t,z)=\E_{t,z}\Big[\exp\Big(-\frac12\int_t^T f(s,Z_s)\ud s-\frac12 g(Z_T)\Big)\Big],
\]
for $(t,z)\in([t_0,T]\times \R^d)\setminus \cD$.
By the Markov property
\begin{equation}\label{eq:ut0}
u(t,z)=\E_{t,z}\Big[\exp\Big(-\frac12\int_t^{t_0} f(s,Z_s)\ud s\Big)u(t_0,Z_{t_0})1_{\{t_0<\tau_\cD\}}\Big],\quad \text{for $(t,z)\in[0,t_0]\times\R^d$}.
\end{equation}
Assuming $f\in(\overline\cC)$ and $g\in C(\R^d)$, it is not hard to verify that $u(t_0,\cdot)\in C(\R^d\setminus \overline{A}_0)$ and it can be continuously extended to $\R^d\setminus A_0$. Then, with no loss of generality, we could also restrict our attention to the case $t_0=T$, upon replacing $g(z)$ with $-\frac12\ln u(t_0,z)$ in \eqref{eq:ut0}.

\begin{proposition}\label{prop:suffcond3}
Let \eqref{eq:Zcont} and Assumption \ref{ass:ball} hold. Then \eqref{eq:regD} holds for any sequence $(t_n,z_n)_{n\in\N}\subset\cC$ converging to $(t,z)\in\cC$ and for any sequence $(t_n,z_n)_{n\in\N}\subset\cC$ converging to $(t,z)\in\cD^\circ_{t_0}$ with $t_n<t_0$. Hence, $u\in C(\cC)$ and $u(t,z)$ restricted to the set $[0,t_0]\times\R^d$ is continuous in $\cC_{[0,t_0)}\cup\cD^\circ_{t_0}=\cO_{[0,t_0]}\setminus(\{t_0\}\times\partial A_0)$, provided that also $f\in C(\overline\cC)$ and $g\in C(\R^d)$.  
\end{proposition}

\begin{proof}
The second statement in the proposition is a direct consequence of first one. We now prove the first statement.

For $(t,z)\in(t_0,T]\times\R^d$ the result is obvious because $\tau^{t,z}_\cD=\infty$ and therefore $\tau^{t_n,z_n}_\cD=\infty$ for $t_n>t_0$. 
For $(t,z)\in[0,t_0]\times\R^d$ it is clear that $\tau^{t,z}_\cD\in\{t_0,\infty\}$. For any sequence $(t_n,z_n)_{n\in\N}\subset\cC$ converging to $(t,z)$ we can select a convergent subsequence $(t_{n_k},z_{n_k})\subset\cC$ such that, by the same argument that leads to \eqref{eq:liminf} in the proof of Proposition \ref{prop:suffcond2}, we have
\[
\liminf_{k\to\infty}\tau^k_\cD(\omega)\ge \tau^{t,z}_\cD(\omega),
\]
for $\omega\in\Omega_0$ with  $\Q(\Omega_0)=1$.
Then, for $\omega\in\Omega_0$ such that $\tau^{t,z}_\cD(\omega)=\infty$, we conclude that $\lim_{k\to\infty}\tau^k_\cD(\omega)=\tau^{t,z}_\cD(\omega)=\infty$. 

For $\omega\in\Omega_0$ such that $\tau^{t,z}_\cD(\omega)=t_0$ we can assume with no loss of generality that $Z^{t,z}_{t_0}(\omega)\in A_0$, because of Assumption \ref{ass:ball}. In particular, we can find an open ball $B_0\subset A_0$ with $\inf_{y\in\overline B_0}\mathrm{dist}(y,\partial A_0)\ge \eps>0$ and $Z^{t,z}_{t_0}(\omega)\in B_0$. Then, by continuity of paths of $Z$ we can also construct a cylinder $C_\omega\coloneqq[t_0-c,t_0]\times A_0$ with $c=c_\omega>0$ and such that $(s,Z^{t,z}_s(\omega))\in C_\omega$ for all $s\in[t_0-c,t_0]$. Therefore, \eqref{eq:Zcont} implies that for sufficiently large $k$ it must be $(s,Z^{k}_s(\omega))\in C_\omega$ for all $s\in[t_0-c,t_0]$. That implies $\tau^k_\cD(\omega)=t_0$ for all sufficiently large $k$'s. So we have $\lim_{k\to\infty}\tau^k_\cD(\omega)=\tau^{t,z}_\cD(\omega)=t_0$. The same argument as in the final paragraph of the proof of Proposition \ref{prop:suffcond2} shows that \eqref{eq:regD} holds.
\end{proof}

A combination of the techniques from Propositions \ref{prop:suffcond2} and \ref{prop:suffcond3} allows in principle to address more complex settings. We leave this further extensions aside because such analysis is more easily carried out on a case by case basis.
\medskip

\noindent{\bf Acknowledgment}: T.\ De Angelis received partial financial support from EU -- Next Generation EU -- PRIN2022 (2022BEMMLZ) and PRIN-PNRR2022 (P20224TM7Z). E.\ Ekström greatfully acknowledges support from the Swedish Research Council.

\vspace{-5pt}

\end{document}